\documentclass[a4paper,11pt]{scrartcl}

\usepackage[utf8]{inputenc}
\usepackage[T1]{fontenc}
\usepackage[final]{graphicx}
\usepackage{float}
\usepackage{mathpazo}
\usepackage{booktabs} 
\usepackage{array} 
\usepackage{paralist} 
\usepackage{verbatim} 
\usepackage{subfig}
\usepackage{enumitem}
\usepackage{amsmath}
\usepackage{amssymb}
\usepackage{amsthm}
\usepackage{verbatim}
\usepackage{hyperref}
\usepackage{cases}
\usepackage{verbatim}
\usepackage{mathtools}
\usepackage[a4paper,lmargin={2cm},rmargin={2cm},
tmargin={2.0cm},bmargin = {2.0cm}]{geometry}
\usepackage{float}
\usepackage{subfig}

\newcommand{\R}{\mathbb{R}}
\newcommand{\HH}{\mathcal{H}}
\newcommand{\GG}{\mathcal{G}}

\DeclareMathOperator{\ri}{ri}
\DeclareMathOperator{\sri}{sri}

\DeclareMathOperator{\dom}{dom}
\DeclareMathOperator*{\argmin}{argmin}

\DeclareMathOperator*\prox{Prox}%

\newtheorem{theorem}{Theorem}
\newtheorem{definition}[theorem]{Definition}
\newtheorem{lemma}[theorem]{Lemma}

\newtheorem{corollary}[theorem]{Corollary}

\newtheorem{remark}[theorem]{Remark}

\newtheorem{example}[theorem]{Example}

\begin{document}

\title{A Dynamical Approach to Two-Block Separable Convex Optimization Problems with Linear Constraints}

\author{Sandy Bitterlich \thanks{Chemnitz University of Technology, Faculty of Mathematics, 09126 Chemnitz, Germany,
email: sandy.bitterlich@mathematik.tu-chemnitz.de. Research supported by DFG (Deutsche Forschungsgemeinschaft), project WA922/9-1.} \and
Ern\"{o} Robert Csetnek \thanks {University of Vienna, Faculty of Mathematics, Oskar-Morgenstern-Platz 1, A-1090 Vienna, Austria,
email: robert.csetnek@univie.ac.at. Research supported by FWF (Austrian Science Fund), project P 29809-N32.}
 \and Gert Wanka \thanks{Chemnitz University of Technology, Faculty of Mathematics, 09126 Chemnitz, Germany,
email: gert.wanka@mathematik.tu-chemnitz.de. Research supported by DFG (Deutsche Forschungsgemeinschaft), project WA922/9-1.}}

\maketitle

\noindent \textbf{Abstract.} The aim of this manuscript is to approach by means of first order
differential equations/inclusions convex programming problems
with two-block separable linear constraints and objectives, whereby (at least) one of the components of the latter is assumed to be strongly convex. Each block of the objective contains a further smooth convex function.
We investigate the dynamical system proposed and prove that its trajectories asymptotically converge to a saddle point of the Lagrangian of the convex optimization problem. Time discretization of the dynamical system
leads to the alternating minimization algorithm AMA and also to its proximal variant recently
introduced in the literature.\vspace{1ex}

\vspace{1ex}


\section{Introduction and preliminaries}

Since the seventies of the last century the investigation of dynamical systems approaching monotone inclusions and optimization problems enjoy much attention (see Br\'{e}zis, Baillon and Bruck, Crandall and Pazy \cite{brezis, baillon-brezis1976, bruck75, crpa69}). This is due to their intrinsic importance in areas like differential equations and applied functional analysis, and also since they have been recognized as a valuable tool for deriving  and investigating numerical schemes for optimization problems obtained by time discretization of the continuous dynamics. The dynamic approach to iterative methods in optimization can furnish deep insights into the expected behavior of the method and the techniques used in the continuous case can be adapted to obtain results for the discrete algorithm. We invite the reader to consult \cite{peyp-sorin2010} for more insights into the relations between the continuous and discrete dynamics.

This research area attracts the attention of the community continuously. There are several works in the last years concerning dynamical systems, which have a connection to numerical algorithms. 
Motivated by the applications in optimization where nonsmooth functions are involved, many authors consider dynamical systems defined via proximal evaluations. Through explicit time discretization they transform in relaxed versions of proximal point algorithms. For example \cite{abbas-att-arx14} Abbas and Attouch proposed a dynamical system  which is a continuous version of the forward backward algorithm (we mention here also the works
of Bolte \cite{bolte-2003} and Antipin \cite{antipin}), in \cite{babo18} an implicit forward-backward-forward dynamical system was introduced and in \cite{csma19} a dynamical system of Douglas-Rachford type was proposed.
Acceleration of the dynamics can be achieved by considering second order differential equations/inclusions where
again resolvents and proximal operators are involved in the description of the systems (see for example \cite{b-c-sicon2016} and
the works of Attouch and his co-authors \cite{att-cab, att-peyp}). This is a flourishing area in the
continuous setting since the work of Su-Boyd-Cand\`{e}s \cite{subo16}, where a second-order ordinary differential equation was proposed as the limit of Nesterov’s accelerated gradient method which involves inertial type schemes.

Let us underline that approaching optimization problems where compositions with linear operators are involved by
means of differential equations/inclusions is relatively new in the literature (and this is the focus also in this manuscript). We mention here \cite{bocs19} (which is related to continuous counterparts of primal-dual
algorithms, Proximal ADMM and the linearized proximal method of multipliers) and also the contribution of Attouch \cite{att} (related to some fast inertial Proximal ADMM schemes).

Before we introduce the dynamical system we want to investigate, let us make precise the optimization
problem we consider and mention some notations used in this context.

We consider the following two-block separable optimization problem:
\begin{align}\label{opt:Prox-AMA:primal}
&\min_{x \in \HH, z \in \GG} f(x)+h_1(x)+g(z)+h_2(z) \quad \text{s.t.} \quad Ax+Bz=b,
\end{align}
where $\mathcal{H}$, $\GG$ and $\mathcal{K}$ are real Hilbert spaces, $f:\HH \to \overline{\R}:=\R \cup \{\pm \infty\}$ is a proper, lower semicontinuous and $\sigma$-strongly convex function with $\sigma>0$ (i.e. $f-(\sigma/2)\|\cdot\|^2$ is convex),
 $g :\GG \to \overline{\R}$ is proper, convex and lower semicontinuous, $h_1:\HH\to \R$ and $h_2:\GG \to \R$ are convex and Fr\'echet differentiable functions with $L_{h_1}$, respectively $L_{h_2}$-Lipschitz continuous gradients ($L_{h_1}\geq 0 $, $L_{h_2} \geq 0$), i.e. $\|\nabla h_1(x)- \nabla h_1(y) \| \leq L_{h_1} \|x-y\|$ for every $x, y \in \HH$ (analogously for $h_2$)
 and $A:\HH \to \mathcal{K}$ and $B:\GG \to \mathcal{K}$
 are linear continuous operators such that $A \neq 0$ and $b \in \mathcal{K}$.

The Lagrangian associated with the optimization problem \eqref{opt:Prox-AMA:primal} is defined by
$L :\HH\times \GG \times \mathcal{K} \to \overline{\R},$
\begin{align*}
L(x,z,y)=f(x)+h_1(x)+g(z)+h_2(z)+\langle y, b-Ax-Bz \rangle.
\end{align*}
We say that $(x^*,z^*,y^*) \in \HH \times \GG \times \mathcal{K}$ is a saddle point of the Lagrangian $L$, if
\[L(x^*,z^*,y) \leq L(x^*,z^*,y^*) \leq L(x,z,y^*) \quad \forall (x,z,y) \in \HH\times \GG \times \mathcal{K}.	\]
It is well-known that $(x^*,z^*,y^*)$ is a saddle point of the Lagrangian $L$ if and only if $(x^*,z^*)$ is an optimal solution of \eqref{opt:Prox-AMA:primal},
$y^*$ is an optimal solution of its Fenchel-Rockafellar dual problem
\begin{equation}\label{Fenchel-dual} \sup_{y\in \mathcal{K}}\{-(f^*\Box h_1^*)(A^*y)-(g^*\Box h_2^*)(B^*y)+\langle y,b\rangle\},
\end{equation}
and the optimal objective values of \eqref{opt:Prox-AMA:primal} and \eqref{Fenchel-dual} coincide.
Note that the (Fenchel) conjugate function $f^*:\mathcal{H} \to \overline{\mathbb{R}}$ of $f:\HH \to \overline{\R}$ is defined as
\begin{equation*}
	f^*(y)=\text{sup}_{x \in \HH}\{\langle y, x \rangle -f(x)  \}  \quad \forall y \in \HH.
\end{equation*}
If $f$ is a proper, convex and lower semicontinuous function, then  $f^{**}=f$, where $f^{**}$ is the conjugate function of $f^*$.
The infimal convolution of two proper functions $f_1,f_2:{\cal H}\rightarrow \overline{\mathbb{R}}$ is the function $f_1\Box f_2:{\cal H}\rightarrow\overline{\mathbb{R}}$,
defined by $(f_1\Box f_2)(x)=\inf_{y\in {\cal H}}\{f_1(y)+f_2(x-y)\}$.

The existence of saddle points for $L$ is guaranteed when \eqref{opt:Prox-AMA:primal} has an optimal solution and, for instance,
the Attouch-Br\'{e}zis-type condition
\begin{equation}\label{reg-cond} b\in\sri (A(\dom f)+B(\dom g))
\end{equation}
holds (see for example \cite[Theorem 3.4]{bot10}). In the finite dimensional setting this asks for the existence of
$x \in \ri(\dom f )$ and $z \in \ri(\dom g)$ satisfying $Ax+Bz=b$. For more on these generalized
interiority notions and their role in optimization we refer the reader to \cite{baco17} and \cite{bot10}.


Let $f$ be a proper, convex and lower semicontinuous function. Then the \emph{Proximal Point Operator} of $f$ with parameter $\gamma >0$ is defined as:
\begin{equation*}
\text{Prox}_{\gamma f}(x)=\argmin_{y \in \mathcal{H}}\left\{\gamma f(y)+\frac{1}{2}\|y-x\|^2\right\}.
\end{equation*}

The system of optimality conditions for the primal-dual pair of optimization problems \eqref{opt:Prox-AMA:primal}-\eqref{Fenchel-dual} reads:
\begin{equation}\label{opt-cond} A^*y^*-\nabla h_1(x^*) \in \partial f(x^*), \
B^*y^*-\nabla h_2(z^*)\in \partial g(z^*) \ \mbox{ and } Ax^*+Bz^*=b.
\end{equation}
More precisely, if \eqref{opt:Prox-AMA:primal} has an optimal solution $(x^*,z^*)$ and a qualification condition, like for instance \eqref{reg-cond}, is fulfilled,
then there exists an optimal solution $y^*$ of \eqref{Fenchel-dual} such that \eqref{opt-cond} holds; consequently, $(x^*,z^*,y^*)$ is a saddle point of the
Lagrangian $L$. Conversely, if $(x^*,z^*,y^*)$ satisfies relation \eqref{opt-cond}, then $(x^*,z^*)$ is an optimal solution of \eqref{opt:Prox-AMA:primal} and
$y^*$ is an optimal solution of \eqref{Fenchel-dual}.
We recall that the convex subdifferential of $f$ is defined as $\partial f(x)=\{u\in \mathcal{H}: f(y)\geq f(x)+\langle u,y-x\rangle \forall y\in \mathcal{H}\}$, if $f(x) \in \R$,
and as $\partial f(x) = \emptyset$, otherwise. Notice that in case $f$ is $\sigma$-strongly convex ($\sigma>0$),  $\partial f$ satisfies the strong monotonicity property: $\langle u-v,x-y\rangle\geq \sigma\|x-y\|^2$ for all
$u\in\partial f(x), v\in\partial f(y)$, see for example \cite{baco17}.

\begin{remark}\label{1remark-sec3} If $(x_1^*,z_1^*,y_1^*)$ and $(x_2^*,z_2^*,y_2^*)$ are two saddle points of the Lagrangian $L$, then $x_1^*=x_2^*$.
	This follows easily from \eqref{opt-cond}, by using the strong monotonicity of $\partial f$ and the monotonicity of $\partial g$.
\end{remark}

Further, we denote by $S_+(\HH)$ the set of operators from $\HH$ to $\HH$ which are linear, continuous, self-adjoint and positive semidefinite.
For $M \in S_+(\HH)$ we define the seminorm $\|\cdot\|_{M} : \HH \rightarrow [0,+\infty)$,  $\|x\|_M= \sqrt{\langle x, Mx \rangle}$.
We consider the Loewner partial ordering on $S_+(\HH)$, defined for $M_1, M_2 \in \mathcal{S}_+(\mathcal{H})$ by
\begin{equation}\label{Loewner}M_1 \succcurlyeq M_2 \Leftrightarrow \|x\|_{M_1} \geq \|x\|_{M_2} \quad \forall x \in \mathcal{H}.
\end{equation}
Furthermore, we define for $\alpha>0$ the set $\mathcal{P}_\alpha(\HH):=\{M\in \mathcal{S}_+(\HH): M \succcurlyeq \alpha \textrm{Id} \}$, where $\textrm{Id} : \HH \rightarrow \HH, \textrm{Id}(x) = x$ for all $x \in \HH$, denotes the identity operator on $\HH$.

Let $A:\HH \to \GG$ be a linear continuous operator. The operator $A^*:\GG \to \HH$, fulfilling $\langle A^*y,x \rangle = \langle y, Ax \rangle $ for all $x \in \HH$ and $y \in \GG$, denotes the adjoint operator of $A$, while
$\|A\|:=\sup\{\|Ax\| : \|x\| \leq 1 \}$ denotes the norm of $A$.

The dynamical system we propose and investigate in this paper is:

\begin{equation}\label{eq:proxAMA-DS}
\begin{cases}
\dot{x}(t)+x(t) \in \left(\partial f + M_1(t)\right)^{-1} \left[M_1(t)x(t)+A^*y(t)-\nabla h_1(x(t))\right]\\[2ex]
\dot{z}(t)+z(t)\in \left(\partial g +  c(t)B^*B+M_2(t)\right)^{-1}\left[M_2(t)z(t)+B^*y(t)-c(t)B^*A(\dot{x}(t)+x(t))\right.\\[2ex]
\left.\quad \quad \quad \quad \quad \quad+c(t)B^*b-\nabla h_2(z(t))\right] \\[2ex]
\dot{y}(t)=c(t)\left(b-A(x(t)+\dot{x}(t))-B(z(t)+\dot{z}(t))\right)\\[2ex]
x(0)=x^0\in \HH, z(0)=z^0 \in \GG, y(0)=y^0 \in \mathcal{K},
\end{cases}
\end{equation}
where $c(t)>0$ for all $t \in [0,+\infty)$, and $M_1:[0,+\infty) \to S_+(\HH)$ and $M_2:[0,+\infty) \to S_+(\GG)$.

In the next section we stress that the dynamical system leads through explicit time discretization to the proximal AMA algorithm \cite{bibo19} and the AMA numerical scheme \cite{tseng91}. Furthermore we underline the role
of the operators $M_1$ and $M_2$, namely for a special choice of the linear maps $M_1$ and $M_2$ we obtain a dynamical system of primal-dual type which is a full splitting scheme. For this we consider a numerical example in order to show how the parameters for these particular linear maps can be chosen and influence the convergence of the trajectories.

We continue with the existence and uniqueness of strong global solutions of the dynamical system proposed above.
The study relies on classical semigroup theory, showing that the system corresponds in fact to a Cauchy-Lipschitz system in a product space. This is far from being trivial and requires several technical prerequisites which are
described in detail.

The last section is devoted to the asymptotic analysis of the trajectories and the connection to the
optimization problems \eqref{opt:Prox-AMA:primal} and \eqref{Fenchel-dual}. The analysis relies
on Lyapunov theory where the derivation of an appropriate energy functional plays a central role. The way
the Lyapunov functional is obtained is quite involved and technical issues have to be investigated in order to
achieve this goal (see the proof of Theorem \ref{PAMADS:th:convergence} and \eqref{PAMADS:eq:sum3a}). Finally, we prove that the trajectories converge weakly to a saddle point of the Lagrangian $L$. We conclude the paper with some open questions and perspective.

The analysis used in this manuscript relies on similar tools considered in \cite{bocs19}. Let us underline some
differences in comparison to \cite{bocs19}. First of all, our optimization problem  \eqref{opt:Prox-AMA:primal} has
a different structure, with two linear operators involved in the constrained set. Second, our dynamical system
is related to the Proximal AMA algorithm \cite{bibo19}, the AMA numerical scheme \cite{tseng91} and primal dual-type
algorithms obtained in \cite{bibo19}. The one in \cite{bocs19} is related to the Proximal ADMM \cite{babo18}, the
classical ADMM and primal-dual type algorithms. Moreover, notice that in our case $f$ is strongly convex which has an influence in the investigations performed here (and in particular the inclusion corresponding to $f$ has a more
tractable form). Moreover, in our analysis we have an additional parameter $c$, which is time varying, and this makes the investigation more involved (taking variable $c$ is motivated by \cite{tseng91}, where the numerical scheme AMA also involves a variable parameter). For more on the AMA algorithm introduced by Tseng and motivation for considering this setting we refer the reader to \cite{tseng91, bibo19, g2014}.

\section{Solution concept, discretizations, example}

We need the following definition before we specify what do we mean by a solution of \eqref{eq:proxAMA-DS}.

\begin{definition}
A function $x:[0,+\infty) \to \HH$ is said to be locally absolutely continuous, if it is absolutely continuous on every interval $[0,T], T>0$; that is, for every $T>0$ there exists an integrable function $y:[0,T] \to \HH$ such that
\[
x(t)=x(0)+\int_{0}^{t}y(s)ds \quad \forall t \in [0,T].
\]
\end{definition}

Notice that every locally absolutely continuous function is differentiable almost everywhere. Moreover,
 $x:[0,T]\to \HH$ is absolutely continuous if and only if (see \cite{att-sv2011, abat14}): for every $\varepsilon > 0$ there exists $\eta >0$ such that for any finite family of intervals $I_k=(a_k,b_k)\subseteq [0,T]$ the following property holds:
$$ \mbox{for any subfamily of disjoint intervals} \ I_j \ \mbox{with} \ \sum_j|b_j-a_j| < \eta \ \mbox{it holds} \ \sum_j\|x(b_j)-x(a_j)\| < \varepsilon.$$

We are now ready to consider the following solution concept.

\begin{definition}
	Let $(x^0,z^0,y^0) \in \HH \times \GG \times \mathcal{K}$, and $M_1:[0,+\infty) \to S_+(\HH)$ and $M_2:[0,+\infty) \to S_+(\GG)$. The function $(x,z,y):[0,+\infty) \to  \HH \times \GG \times \mathcal{K}$ is called a strong global solution of \eqref{eq:proxAMA-DS}, if the following properties are satisfied:
	\begin{enumerate}
		\item the functions $x,z,y$ are locally absolutely continuous,
		\item for almost every $t \in [0,+\infty)$
		\begin{align*}
		\dot{x}(t)+x(t) \in &\left(\partial f + M_1(t)\right)^{-1} \left[M_1(t)x(t)+A^*y(t)-\nabla h_1(x(t))\right]\\
		\dot{z}(t)+z(t)\in &\left(\partial g +  c(t)B^*B+M_2(t)\right)^{-1}\left[M_2(t)z(t)+B^*y(t)-c(t)B^*A(\dot{x}(t)+x(t))\right.\\
		&\left.+c(t)B^*b-\nabla h_2(z(t))\right] \\
		\dot{y}(t)=&c(t)\left(b-A(x(t)+\dot{x}(t))-B(z(t)+\dot{z}(t))\right),
		\end{align*}
		\item $x(0)=x^0, z(0)=z^0$, and $y(0)=y^0$.
	\end{enumerate}
\end{definition}

\begin{remark} Let us consider a discretization of the considered dynamical system.
	The first two inclusions in \eqref{eq:proxAMA-DS} can be written in an equivalent way as
	\begin{align}
	0 \in &~\partial f(\dot{x}(t)+x(t))+M_1(t)\dot{x}(t)-(A^*y(t)-\nabla h_1(x(t))),\label{PAMADS:OC_DS_1}\\
	0 \in &~\partial g(\dot{z}(t)+z(t))+c(t)B^*B(\dot{z}(t)+z(t))+M_2(t)\dot{z}(t)\nonumber\\
	&-(B^*y(t)-c(t)B^*A(\dot{x}(t)+x(t))+c(t)B^*b-\nabla h_2(z(t))),\label{PAMADS:OC_DS_2}
	\end{align}
	$  \forall t \in [0, +\infty)$. Through explicit discretization with respect to the time variable $t$ and constant step size $h_k \equiv 1$ (i.e. $x(t)\approx x_k$ and $\dot x(t)\approx x^{k+1}-x^k$) we obtain for all $k \geq 0$ the inclusions:
	\begin{align*}
	0 \in &~\partial f(x^{k+1})+M_1^k(x^{k+1}-x^k)-A^*y^k+\nabla h_1(x^k),\\
	0 \in &~\partial g(z^{k+1})+c_kB^*B(z^{k+1})+M_2^k(z^{k+1}-z^k)
	-B^*y^k+c_kB^*A(x^{k+1})-c_kB^*b+\nabla h_2(z^k).
	\end{align*}
	Furthermore, using convex subdifferential calculus this can be written equivalently for all $ k \geq 0$ as
	\begin{align*}
	0 \in &~\partial \left( f+ \langle \cdot-x^k, \nabla h_1(x^k)\rangle
	     -\langle y^k,A\cdot \rangle + \frac{1}{2}\|\cdot-x^k\|^2_{M_1^k}\right)(x^{k+1})\\
	0 \in &~\partial \left( g+ \langle \cdot-z^k, \nabla h_2(z^k)\rangle
	     -\langle y^k,B\cdot \rangle +\frac{c_k}{2}\|Ax^{k+1}+B\cdot-b\|^2+ \frac{1}{2}\|\cdot-z^k\|^2_{M_2^k}\right)(z^{k+1})\\
	\end{align*}
	Hence the dynamical system \eqref{eq:proxAMA-DS} provides through explicit time discretization the following numerical algorithm:
	
	Let $M_1^k \in \mathcal{S}_+(\HH)$ and $M_2^k \in \mathcal{S}_+(\GG)$. Choose $(x^0,z^0,y^0) \in \HH \times \GG \times \mathcal{K}$ and $(c_k)_{k \geq 0}>0$. For all $k\geq 0$ generate the sequence $(x^k,z^k,y^k)_{k\geq0}$ as follows:
	\[
	\begin{cases}\label{alg:prox-AMA-DS}
			x^{k+1}=\argmin_{x \in \HH}\left\{f(x)-\langle y^k,Ax\rangle
			+\langle x-x^{k}, \nabla h_1(x^{k}) \rangle
			+\frac{1}{2}\|x-x^{k}\|^2_{M_1^{k}}\right\}\\[2ex]
			z^{k+1} \in  \argmin_{z \in \GG}\left\{g(z)-\langle y^k,Bz\rangle
			+\frac{1}{2}c_k\|Ax^{k+1}+Bz-b\|^2  \right. \\[2ex]
			\qquad \qquad \qquad \quad\left. +\langle z-z^k, \nabla h_2(z^k) \rangle
			+\frac{1}{2}\|z-z^k\|^2_{M_2^{k}}\right\},   \\[2ex]
			y^{k+1}=y^k+c_k(b-Ax^{k+1}-Bz^{k+1}).
	\end{cases}
	\]
	The algorithm above is the proximal AMA algorithm from \cite{bibo19}. Moreover, in the particular case
$M_1^k=M_2^k=0$ and $h_1=h_2=0$, the numerical scheme is the AMA algorithm introduced by Tseng in \cite{tseng91}.
\end{remark}

\begin{remark}\label{PAMADS:re:M_2(t)} Let us show now that an appropriate choice of $M_2$ leads (both in continuous and discrete case) to
an implementable proximal step in the second inclusion. This is crucial for numerical results in applications,
see also \cite{bibo19} and \cite{babo17}. 
	For every $t\in [0,+\infty)$, we define
	\[M_2(t)=\frac{1}{\tau(t)}\text{Id}-c(t)B^*B,\]
	where $\tau(t)>0$ and $\tau(t)c(t)\|B\|^2 \leq 1$.
	
	Let $t\in [0,+\infty)$ be fixed. Then $M_2(t)$ is positively semidefinite, and the second relation in the dynamical system \eqref{eq:proxAMA-DS} becomes a proximal step. Indeed, under the given conditions, one can see that \eqref{PAMADS:OC_DS_2} is equivalent to
	\begin{align*}
		\left(\frac{1}{\tau(t)}\text{Id}-c(t)B^*B\right)z(t)+B^*y(t)-c(t)B^*A(\dot{x}(t)+x(t))+c(t)B^*b-\nabla h_2(z(t)) \in \\
		 \frac{1}{\tau(t)}\dot{z}(t)+\frac{1}{\tau(t)}z(t)+\partial g(\dot{z}(t)+z(t)).
	\end{align*}
	It follows that
	\begin{align*}
	\dot{z}(t)+z(t)=(\text{Id}+\tau(t) \partial g)^{-1}& ((\text{Id}-\tau(t)c(t)B^*B)z(t) + \tau(t)B^*y(t) -c(t)\tau(t)B^*A(\dot{x}(t)+x(t))\\
	& +c(t)\tau(t)B^*b -\tau(t)\nabla h_2(z(t))),
	\end{align*}
	which is the same as
	\begin{align*}
	\dot{z}(t)+z(t)=\text{Prox}_{\tau(t)g}& ((\text{Id}-\tau(t)c(t)B^*B)z(t) + \tau(t)B^*y(t) -c(t)\tau(t)B^*A(\dot{x}(t)+x(t))\\
	& +c(t)\tau(t)B^*b -\tau(t)\nabla h_2(z(t))).
	\end{align*}
	If we choose furthermore $M_1(t)=0$, our dynamical system \eqref{eq:proxAMA-DS} can be written in this particular setting equivalently as
	\begin{equation}\label{eq:proxAMA-DS_special_case}
	\begin{cases}
	\dot{x}(t)+x(t) \in \left(\partial f\right)^{-1} \left[A^*y(t)-\nabla h_1(x(t))\right]\\[2ex]
	\dot{z}(t)+z(t)=\text{Prox}_{\tau(t)g} ((\text{Id}-\tau(t)c(t)B^*B)z(t) + \tau(t)B^*y(t) -c(t)\tau(t)B^*A(\dot{x}(t)+x(t))\\
	\quad \quad  \quad\quad \quad \quad +c(t)\tau(t)B^*b -\tau(t)\nabla h_2(z(t))). \\[2ex]
	\dot{y}(t)=c(t)\left(b-A(x(t)+\dot{x}(t))-B(z(t)+\dot{z}(t))\right)\\[2ex]
	x(0)=x^0\in \HH, z(0)=z^0 \in \GG, y(0)=y^0 \in \mathcal{K},
	\end{cases}
	\end{equation}
	where $c(t)>0$ for all $t \in [0,+\infty)$. This can be seen as the continuous counterpart with proximal step
of the AMA scheme \cite{tseng91}.
\end{remark}	

\begin{remark}	\label{re:proxAMA-DS_uvw}
	In this paper we will often use the following equivalent formulation of the dynamical system \eqref{eq:proxAMA-DS}. For $U(t)=(x(t),z(t),y(t))$, \eqref{eq:proxAMA-DS} can be written as
		\[
		\begin{cases}
		\dot{U}(t)=\Gamma(t,U(t))\\
		U(0)=(x^0,z^0,y^0)
		\end{cases}
		\]
	where \[\Gamma:[0,+\infty)\times\HH\times\GG\times\mathcal{K}\longrightarrow\HH\times\GG\times\mathcal{K}, \quad \Gamma(t,x,z,y)=(u,v,w),\]
	is defined as
	
	\begin{equation}\label{eq:proxAMA-DS_uvw}\begin{cases}\begin{aligned}
	u=u(t,x,z,y)= &\argmin_{p \in \HH}\left\{F(t,p)
	+\frac{1}{2}\|p-(M_1(t)x+A^*y-\nabla h_1(x))\|^2\right\}-x\\[2ex]
	v=v(t,x,z,y) \in &\argmin_{q \in \GG}\left\{G(t,q)
	+\frac{c(t)}{2}\left\|q-\left(\frac{1}{c(t)}M_2(t)z+\frac{1}{c(t)}B^*y\right.\right.\right.\\[2ex]
	&\left.\left.\left.-B^*A(u+x)+B^*b-\frac{1}{c(t)}\nabla h_2(z))\right)\right\|^2\right\}-z\\[2ex]
	w=w(t,x,z,y)=&c(t)(b-A(x+u)-B(z+v))
	\end{aligned}\end{cases}
	\end{equation}
	with
	\[F:[0,+\infty)\times \HH \to \overline{\R}, \quad F(t,p)=f(p)-\frac{1}{2}\|p\|^2+\frac{1}{2}\|p\|^2_{M_1(t)}\]
	and
	\[G:[0,+\infty)\times \GG \to \overline{\R}, \quad G(t,q)=g(q)+\frac{c(t)}{2}\left(\|Bq\|^2-\|q\|^2\right)+\frac{1}{2}\|q\|^2_{M_2(t)}.\]

	Let $t \in [0,+\infty)$ be fixed. The functions $F(t,\cdot)$ and $G(t,\cdot)$ are proper and lower semicontinuous. When we assume that there exists an $\beta(t)>0$ such that $c(t)B^*B+M_2(t)\in \mathcal{P}_{\beta(t)}(\GG)$, then $q \mapsto G(t,q)+\frac{1}{2}\|q-v\|^2$ is proper, strongly convex and lower semicontinuous  for every $v \in \GG$. Since $M_1(t) \in S_+(\HH)$ and $f$ is strongly convex, the function $p\mapsto F(t,p)+\frac{1}{2}\|p-u\|^2$ is proper, strongly convex and lower semicontinuous for every $u \in \HH$.
	
	Therefore, if the assumption
	\begin{equation*}
	\text{(\textit{Cweak})}\quad \text{for every }t\in [0,+\infty) \text{ there exists an } \beta(t)>0 \text{ such that } c(t)B^*B+M_2(t)\in \mathcal{P}_{\beta(t)}(\GG)
	\end{equation*}
	holds, then in \eqref{eq:proxAMA-DS_uvw} $u$ and $v$ are uniquely defined. 	
	
	A stronger variant of condition (\textit{Cweak}) is
	\begin{equation*}
	\text{(\textit{Cstrong})}\quad \text{ there exists an } \beta>0 \text{ such that } c(t)B^*B+M_2(t)\in \mathcal{P}_{\beta}(\GG) \quad \forall t \in [0,+\infty).
	\end{equation*}
	Note that if (\textit{Cstrong}) holds, then (\textit{Cweak}) holds with $\beta(t):=\beta>0$ for every $t \in [0,+\infty)$.	
\end{remark}	

\begin{example} \label{PAMADS:example}
	We consider the following optimization problem
	\begin{align}
	\inf_{x\in \R^2, z \in \R^2}\frac{1}{2}\left\|x-d\right\|^2+\|z\|_1,\label{PAMADS_example1_primal}\\
	\text{s.t.}  \quad Ax+Bz=0 \nonumber
	\end{align}
	with
	\[
	A=\frac{1}{\sqrt{8}}\begin{pmatrix*}[r]
	2 & 1 \\
	-2 & 1
	\end{pmatrix*} \quad \text{,} \quad
	B=\frac{1}{5}\begin{pmatrix*}[r]
	-3 & 0 \\
	4 & 0
	\end{pmatrix*}\quad \text{and} \quad
	d=\begin{pmatrix*}[r]
	1\\
	0
	\end{pmatrix*},
	\]
	which is problem \eqref{opt:Prox-AMA:primal} with $\HH=\GG=\R^2$, $f,g,h_1,h_2:\R^2\to \R, f(x)=\frac{1}{2}\|x-d\|^2, g(z)=\|z\|_1, h_1(x)=h_2(z)=0,$ for every $x\in\R^2$ and $z\in \R^2$. One can verify that \eqref{PAMADS_example1_primal} has a unique optimal solution, which is $x^*=(0,0)$ and $z^*=(0,0)$. The Fenchel-Rockafellar dual problem of \eqref{PAMADS_example1_primal} is
	\begin{equation*}
		\sup_{y \in \R^2}\{-f^*(A^*y)-g^*(B^*y)\}
	\end{equation*}
	which is equivalent to
	\begin{equation*}
		-\inf_{y \in \R^2}\{f^*(A^*y)+g^*(B^*y)\}.
	\end{equation*}
	and
	\begin{equation}\label{PAMADS_example1_dual}
	-\inf_{\|B^*y\|_{\infty}\leq 1} \{\frac{1}{2}\|A^*y\|^2+\langle A^*y,d \rangle\},
	\end{equation}
	where the unique optimal solution is  $y^*=(0.7071,-0.7071)$.
	
	For $U(t)=(x(t),z(t),y(t))$, $M_1(t)=0$ and $M_2(t)=\frac{1}{\tau(t)}\text{Id}-c(t)B^*B$ we can write the dynamical system for this problem \eqref{eq:proxAMA-DS} similarly as in Remark \ref{re:proxAMA-DS_uvw} (see also
\eqref{eq:proxAMA-DS_special_case})
	\[
	\begin{cases}
	\dot{U}(t)=\Gamma(U(t))\\
	U(0)=(x^0,z^0,y^0)
	\end{cases}
	\]
	where \[\Gamma:[0,+\infty)\times\HH\times\GG\times\mathcal{K}\longrightarrow\HH\times\GG\times\mathcal{K}, \quad \Gamma(t,u_1,u_2,u_3)=(u_4,u_5,u_6),\]
	is defined as
	\begin{equation*}\begin{cases}\begin{aligned}
	u_4= &\argmin_{p \in \HH}\left\{f(p)-\frac{1}{2}\|p\|^2 +\frac{1}{2}\|p-A^*u_3\|^2\right\}-u_1\\
	   = &A^*u_3+d-u_1\\
	u_5= &\text{Prox}_{\tau(t)g}\left((Id-\tau(t)c(t)B^*B)u_2+\tau(t)B^*u_3-c(t)\tau(t)B^*A(u_1+u_4)\right)-u_2\\
	u_6=&c(t)(-A(u_1+u_4)-B(u_2+u_5)).
	\end{aligned}\end{cases}
	\end{equation*}
	We solved the dynamical system with the starting points $x^0=(-10,10), z^0=(-10,10)$ and $y^0=(-10,10)$ in the case when $c(t)>0$ and $\tau(t)>0$ and used the Matlab function \texttt{ode15s}.
	Notice that
	\begin{equation*}
	\prox\nolimits_{\tau(t) g}(x)=x-\tau(t)\text{proj}_{[-1,1]^2}\left(\frac{1}{\tau(t)}x\right)
	\end{equation*}
	where $\text{proj}_{Q}$ is the projection operator on a convex and closed set $Q \subseteq \HH$.
	To assure the convergence of the algorithm we will prove later in Theorem \ref{PAMADS:th:convergence} that it has to be fulfilled for an $\epsilon>0$ that $c(t)<\frac{\sigma}{\|A\|^2}-\epsilon$ for all $t$, where $\sigma$ is the strong convexity parameter of $f(x)$ (here $\sigma=1$), and that $c(t)$ is monotonically decreasing and Lipschitz continuous. For $c(t)=c$ constant, we can choose $c$ such that $c<\frac{2\sigma}{\|A\|^2}-\epsilon$. Besides, it has to be fulfilled that $M_2(t)$ is monotonically decreasing, locally absolutely continuous, positive definite and $\sup_{t\geq 0}\|\dot{M}_2(t)\|<+\infty$ (we are in setting 1. of Theorem \ref{PAMADS:th:convergence}, see also Corollary \ref{PAMADS:co:convergence_special_case} and Remark \ref{rem-ex}).
	
	To guarantee that $M_2(t)$ is positive definite we have to choose $\tau(t)$ such that $\tau(t) c(t)\|B\|^2 < 1$.
	Since $\|A\|^2=1$ and $\|B\|^2=1$ we can choose $c(t) \in  (\epsilon,1-\epsilon)$ and for $c$ constant $c \in (\epsilon,2-\epsilon)$ and $\tau(t) c(t)<1$.
	We considered for $c(t)$ two constant choices, namely, $c(t)=0.25$ and $1.99$ and two variable choices $c_1(t)=\frac{1}{t^2+1.1}+0.01$ 
	and $c_2(t)=\frac{1}{\sqrt{t+1.1}}+0.01$. Furthermore, we chose $\tau(t)c(t)=0.25$ and $0.99$. These parameters fulfill the conditions above. In Figure 1 and 2 below one can see that independent of the choice of $c(t)$  all three trajectories converge faster for a greater value of  $\tau(t)c(t)$. Furthermore one can see that for smaller values of $c(t)$ the trajectories converge faster.\\

\begin{figure}\label{fig1}
	\subfloat{\includegraphics[width=0.33\linewidth]{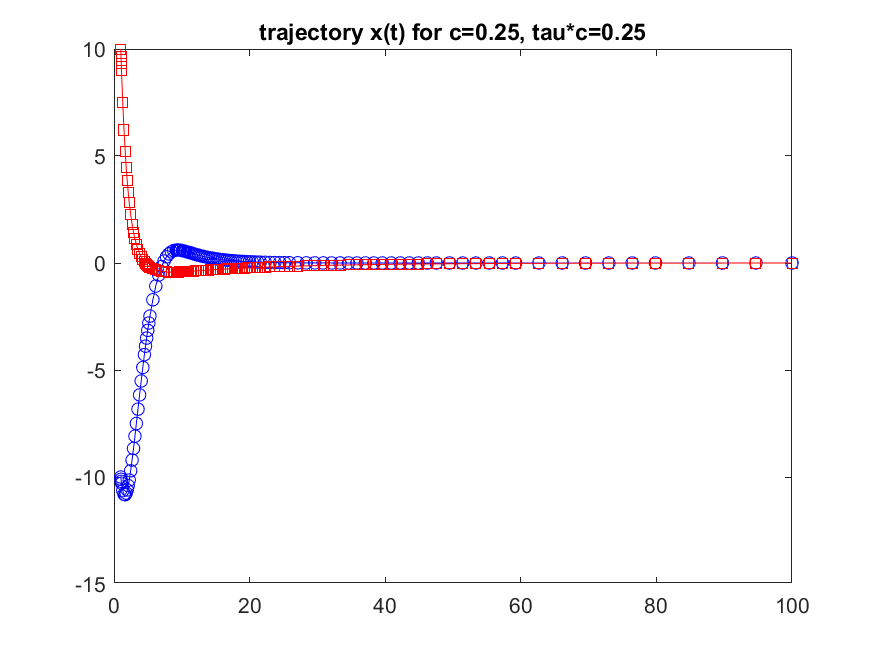}}
	\subfloat{\includegraphics[width=0.33\linewidth]{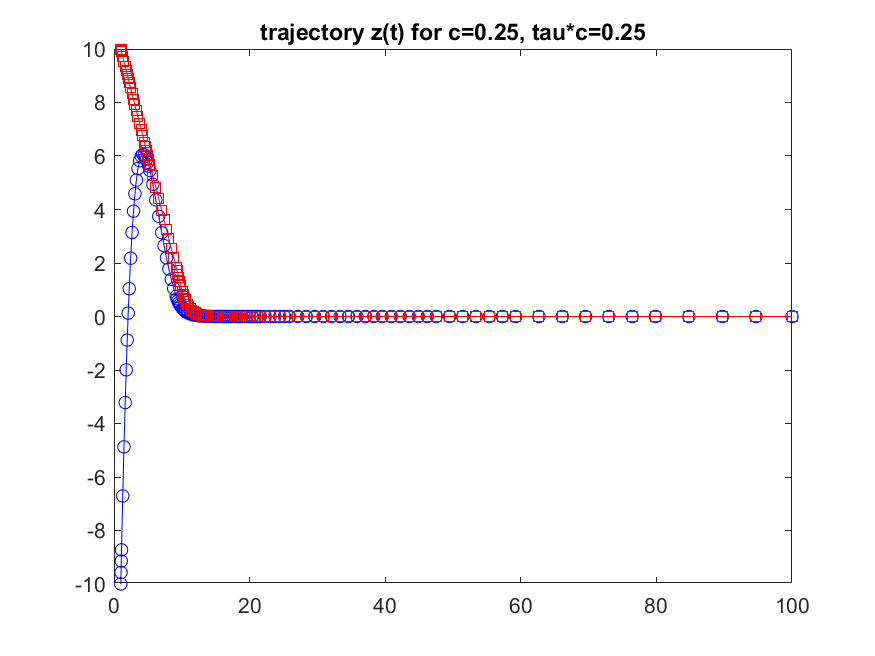}}
	\subfloat{\includegraphics[width=0.33\linewidth]{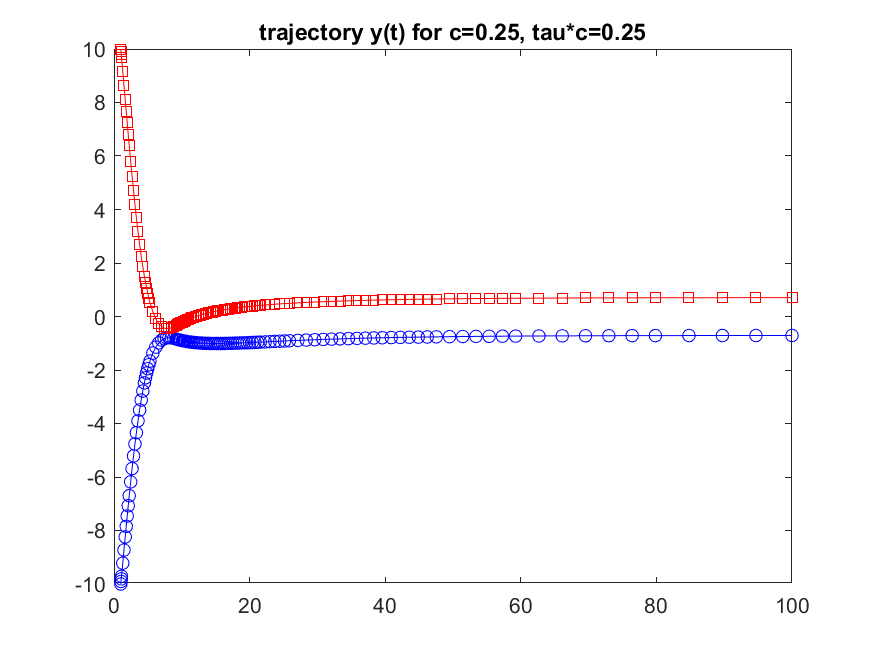}}\\
	\subfloat{\includegraphics[width=0.33\linewidth]{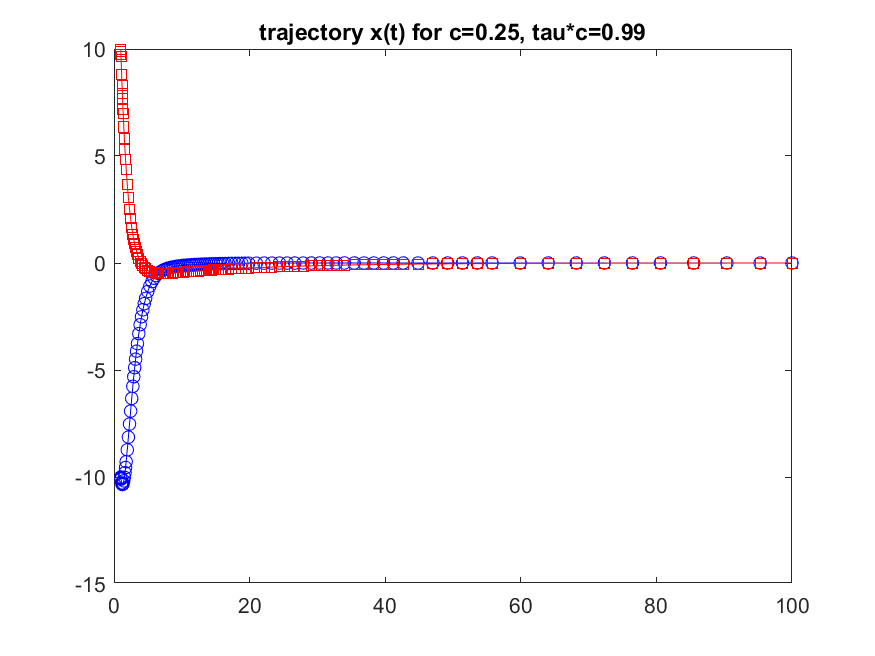}}
	\subfloat{\includegraphics[width=0.33\linewidth]{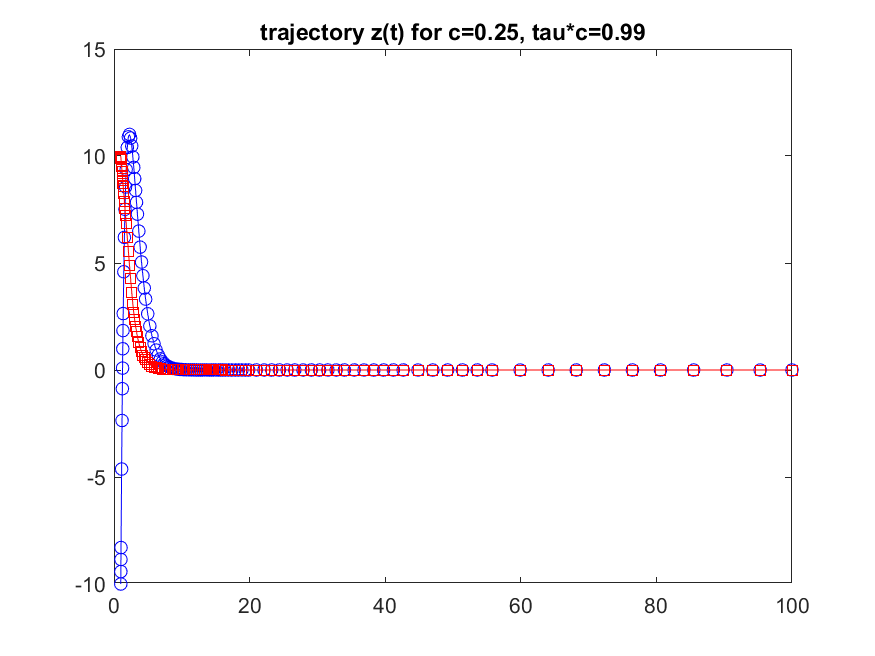}}
	\subfloat{\includegraphics[width=0.33\linewidth]{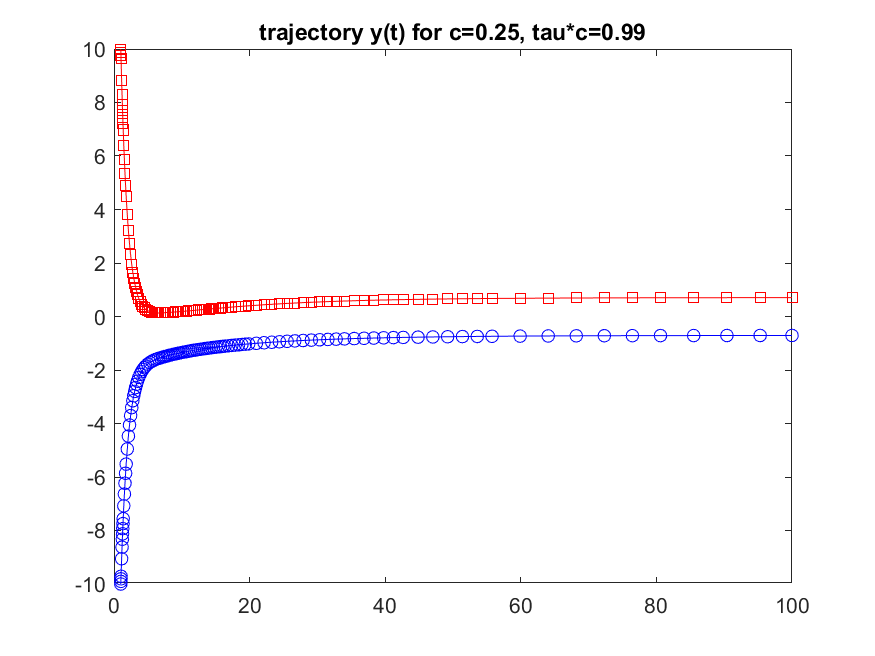}}\\
	\subfloat{\includegraphics[width=0.33\linewidth]{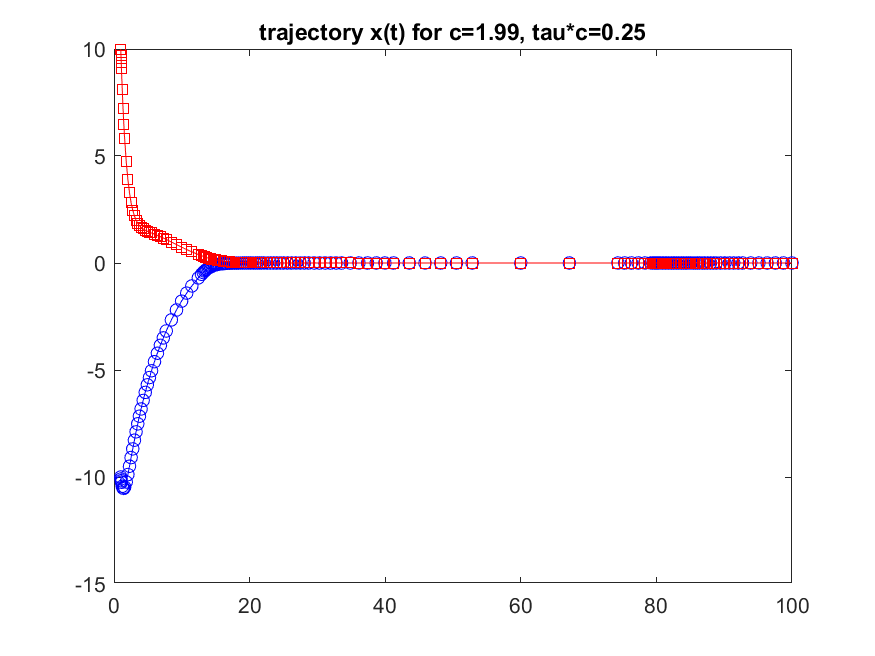}}
	\subfloat{\includegraphics[width=0.33\linewidth]{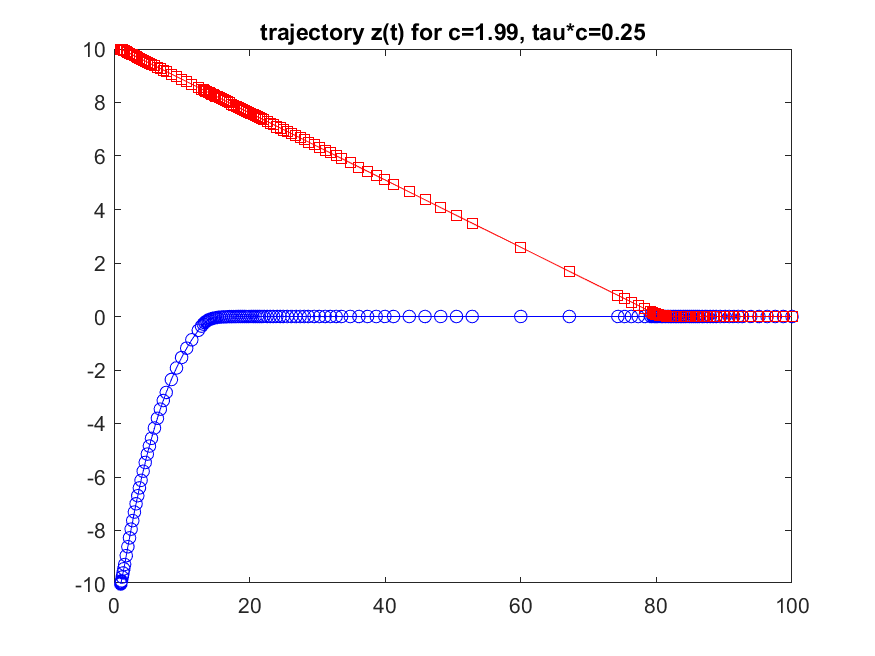}}
	\subfloat{\includegraphics[width=0.33\linewidth]{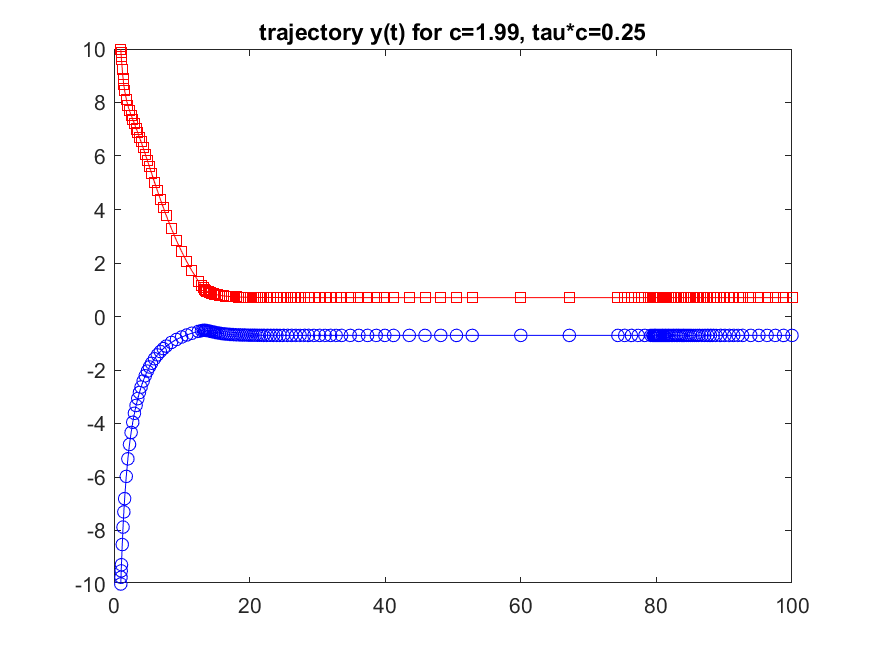}}\\
	\subfloat{\includegraphics[width=0.33\linewidth]{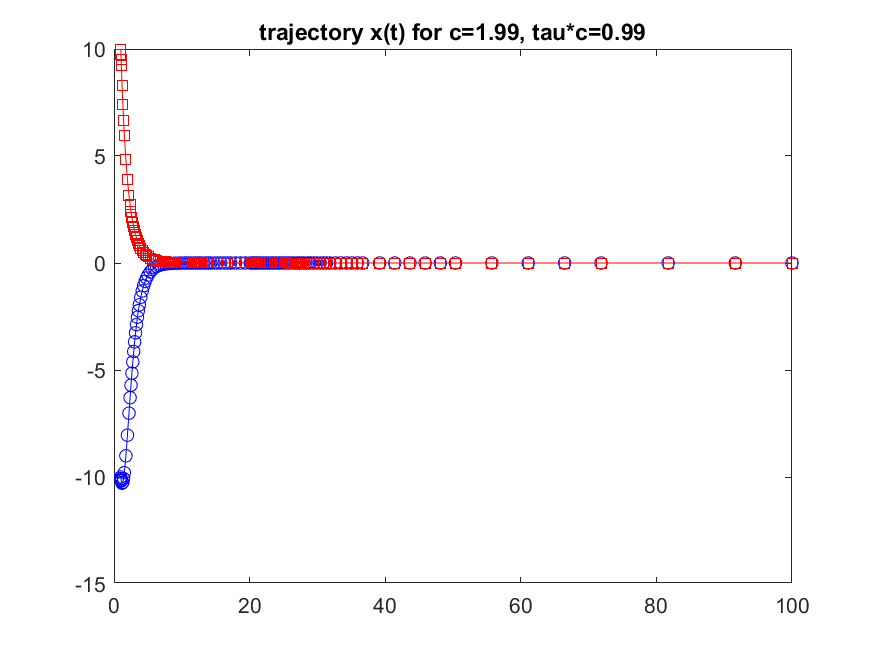}}
	\subfloat{\includegraphics[width=0.33\linewidth]{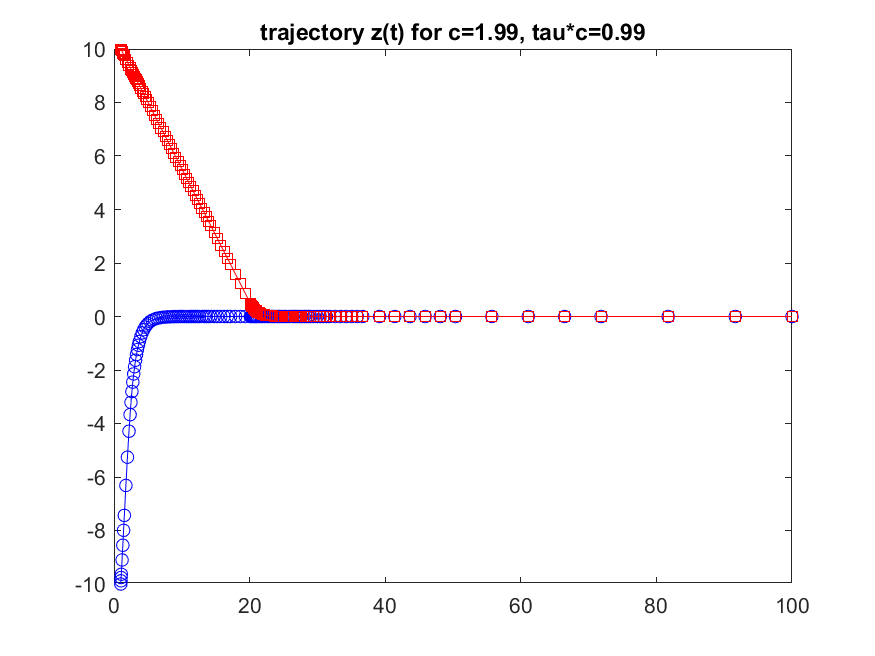}}
	\subfloat{\includegraphics[width=0.33\linewidth]{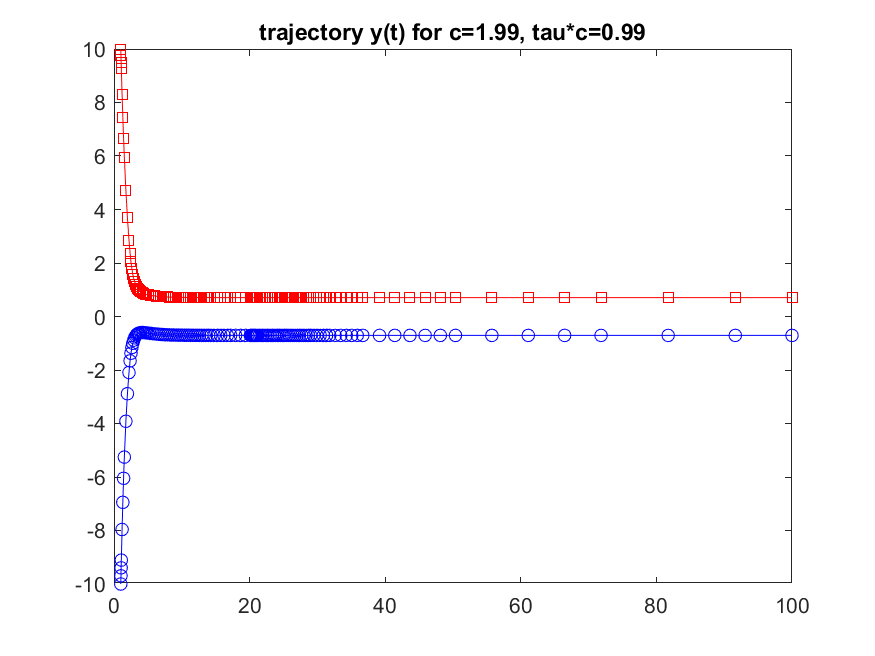}}
	\caption{First and second column: the primal trajectories $x(t)$ and $z(t)$ converge to the primal optimal solution $(0,0)$ for constant $c(t)$ and starting point $(-10,10)$. Third column: the dual trajectory $y(t)$ converges to the dual optimal solution $(0.7071,-0.7071)$ for constant $c(t)$ and starting point $(-10,10)$. }
\end{figure}

\begin{figure}\label{fig2}
	\subfloat{\includegraphics[width=0.33\linewidth]{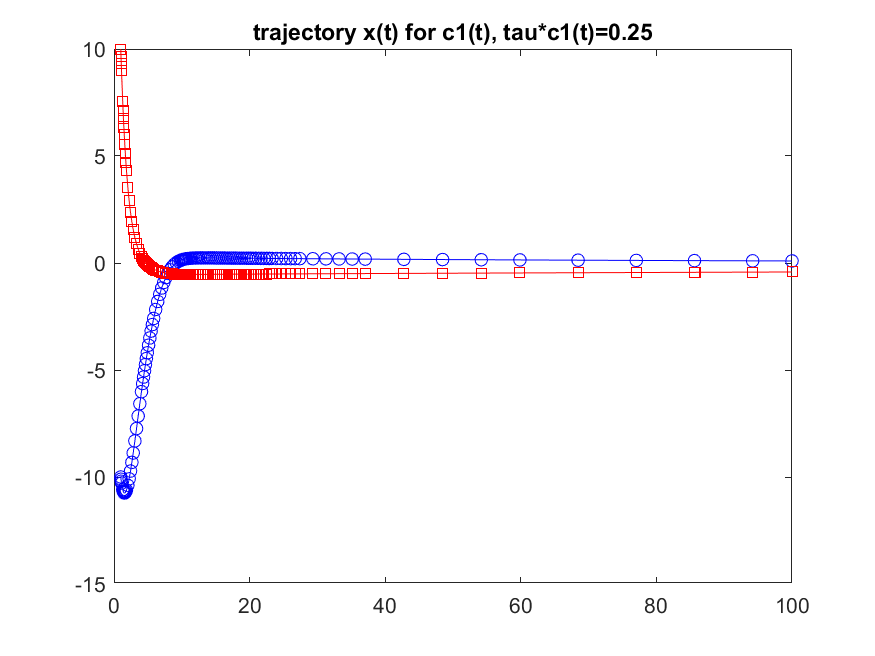}}
	\subfloat{\includegraphics[width=0.33\linewidth]{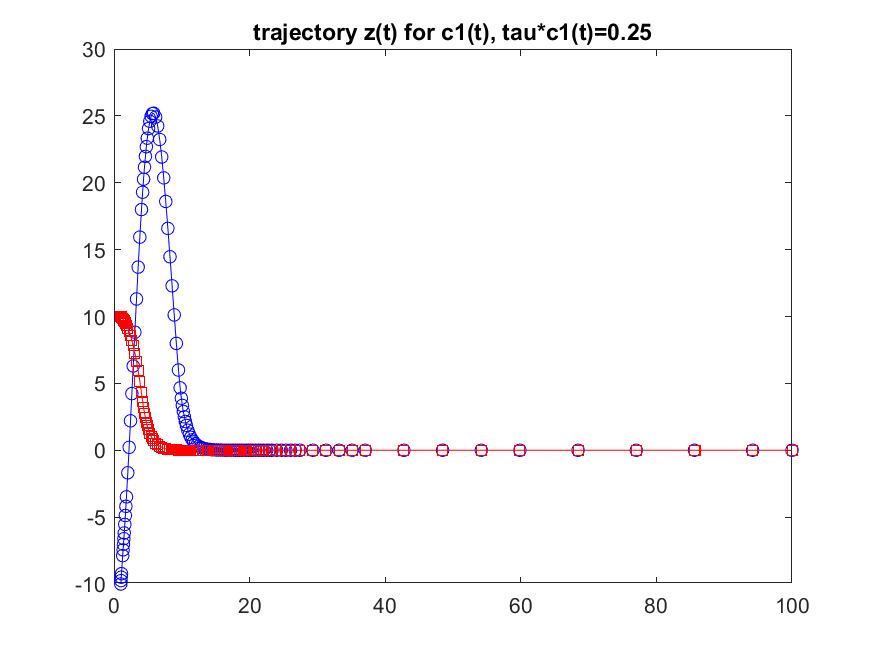}}
	\subfloat{\includegraphics[width=0.33\linewidth]{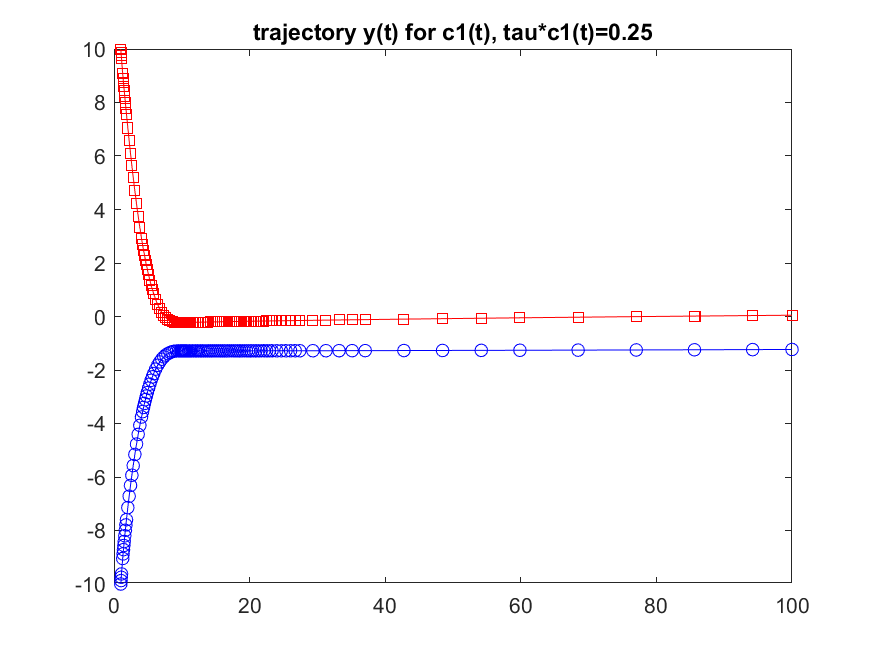}}\\
	\subfloat{\includegraphics[width=0.33\linewidth]{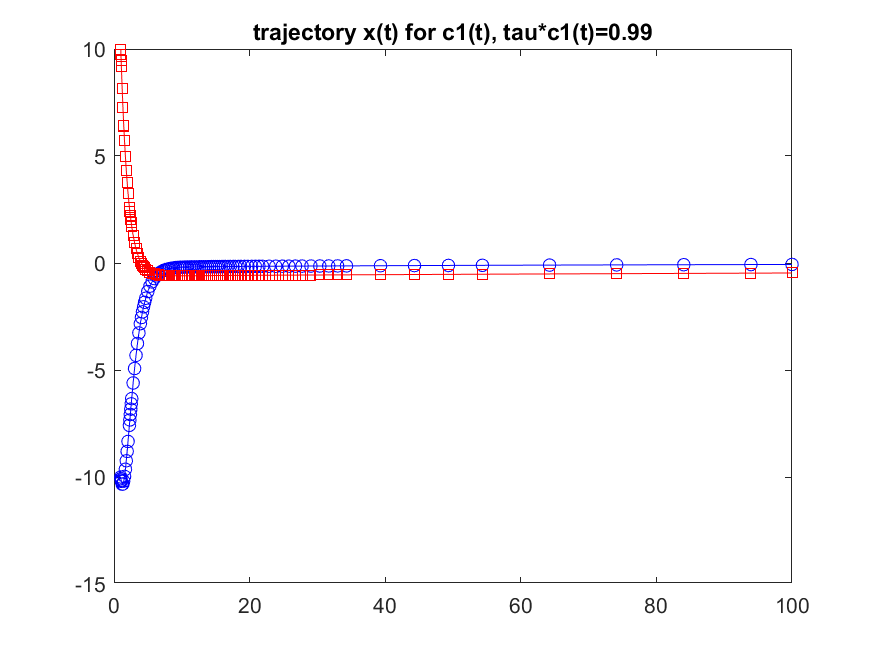}}
	\subfloat{\includegraphics[width=0.33\linewidth]{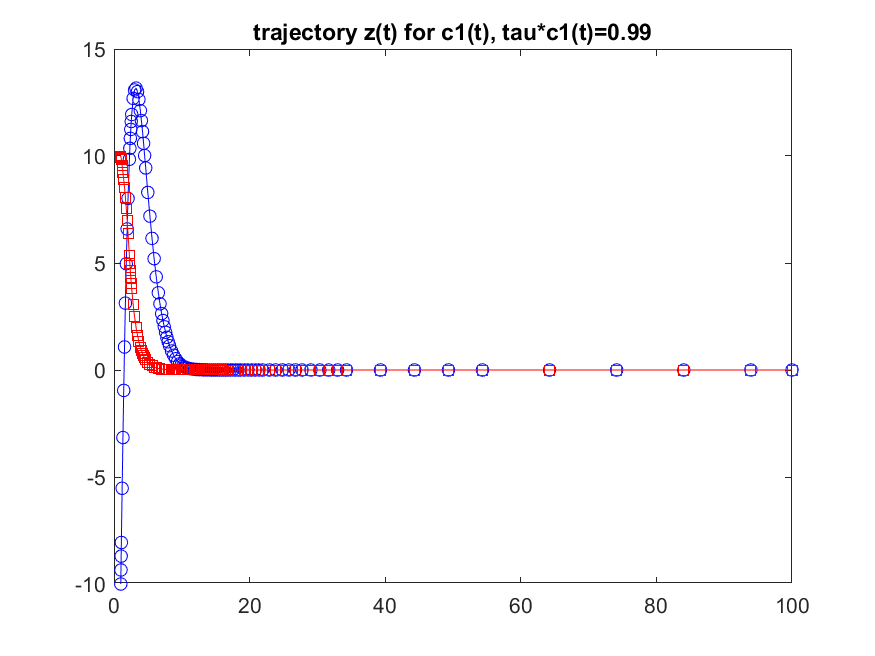}}
	\subfloat{\includegraphics[width=0.33\linewidth]{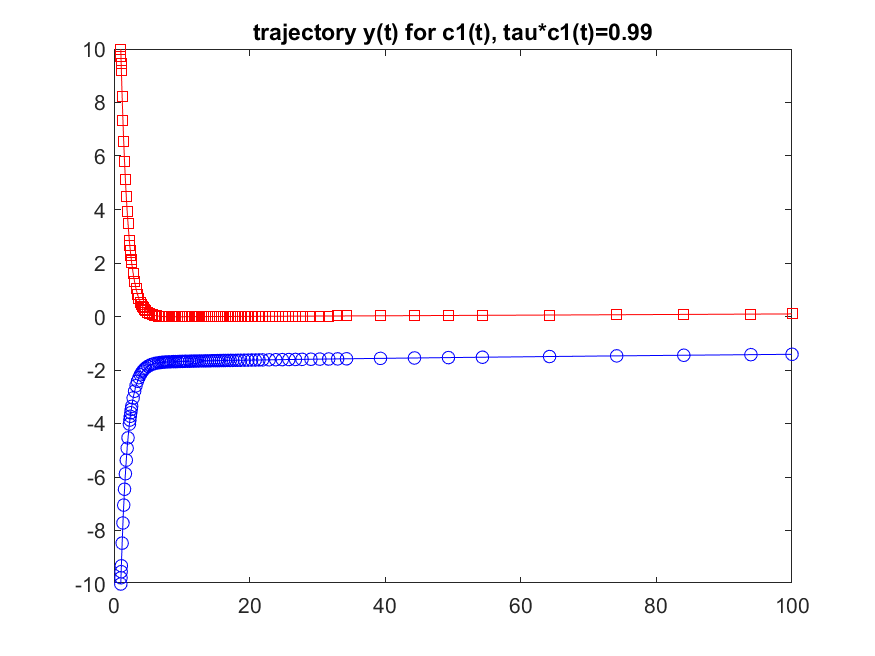}}\\
	\subfloat{\includegraphics[width=0.33\linewidth]{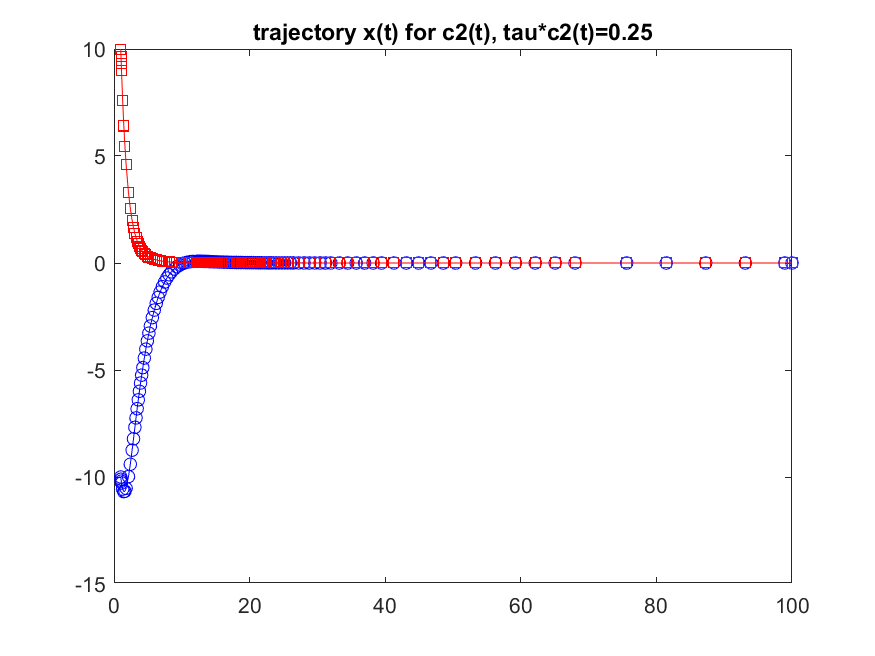}}
	\subfloat{\includegraphics[width=0.33\linewidth]{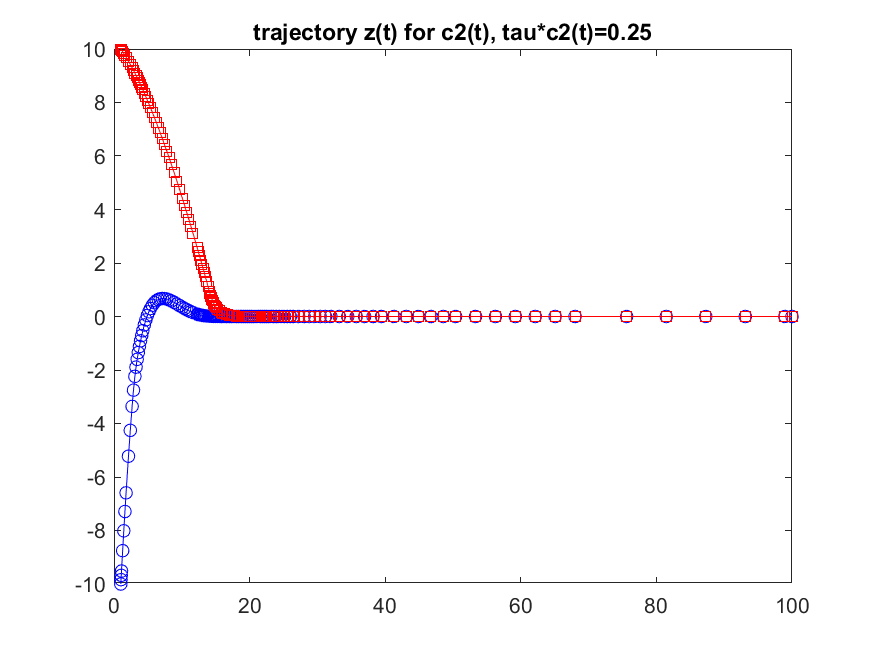}}
	\subfloat{\includegraphics[width=0.33\linewidth]{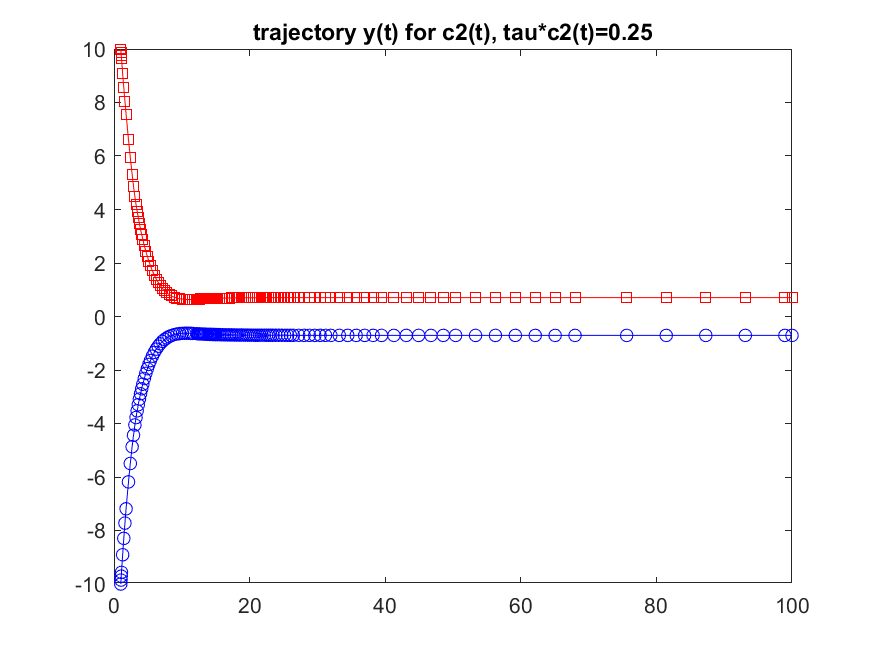}}\\
	\subfloat{\includegraphics[width=0.33\linewidth]{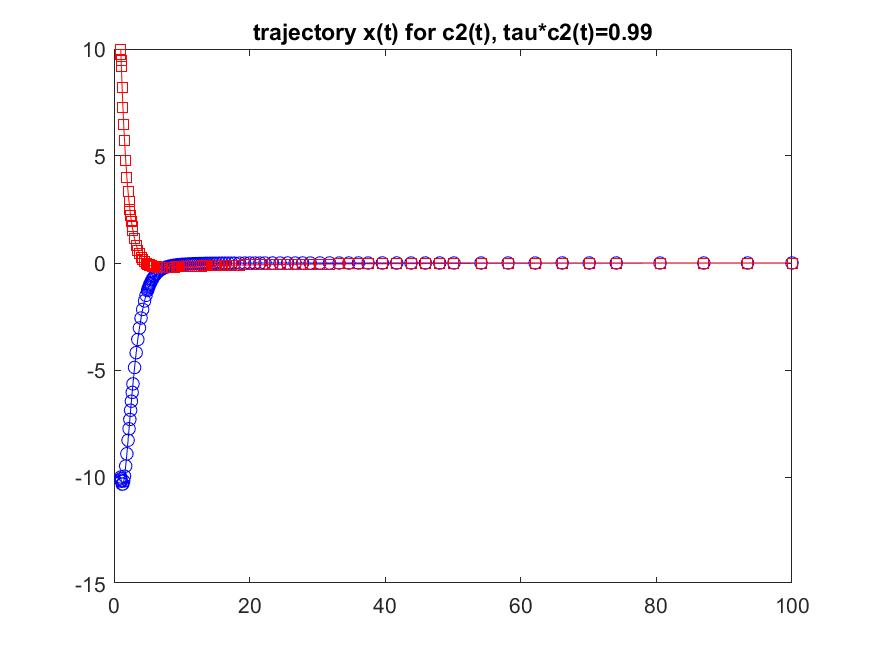}}
	\subfloat{\includegraphics[width=0.33\linewidth]{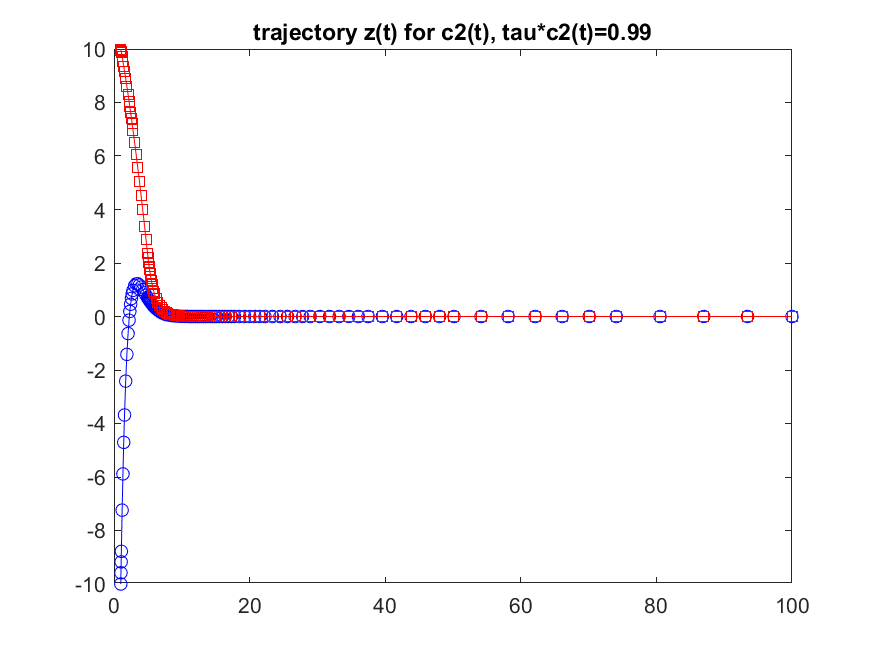}}
	\subfloat{\includegraphics[width=0.33\linewidth]{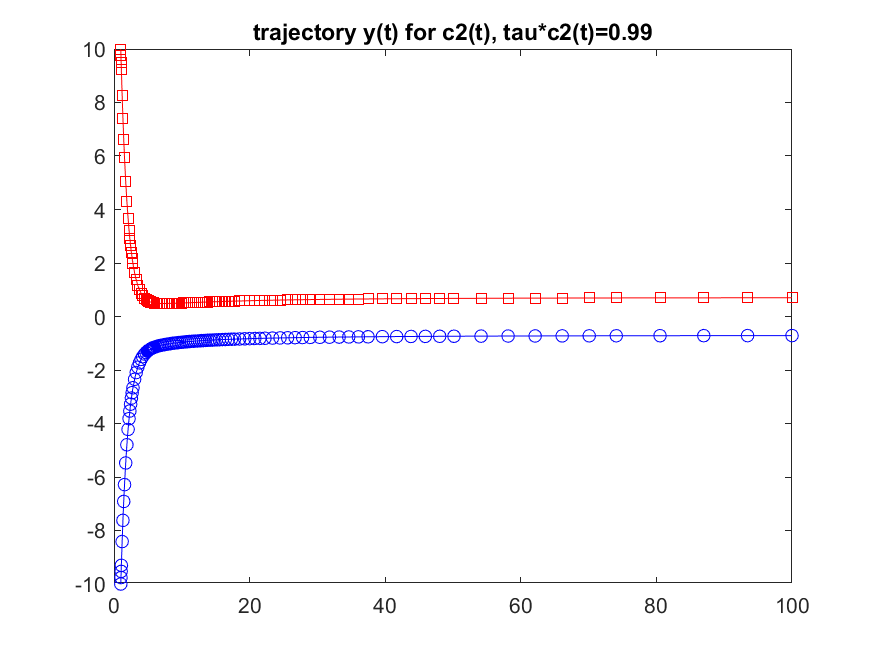}}
	\caption{First and second column: the primal trajectories $x(t)$ and $z(t)$ converge to the primal optimal solution $(0,0)$ for variable $c(t)$ and starting point $(-10,10)$. Third column: the dual trajectory $y(t)$ converges to the dual optimal solution $(0.7071,-0.7071)$ for variable $c(t)$ and starting point $(-10,10)$. }
\end{figure}
	
\end{example}

\section{Existence and uniqueness of the trajectories}
In this section we will investigate the existence and uniqueness of the trajectories generated by the dynamical system \eqref{eq:proxAMA-DS}. We need several preparatory results in order to show that we are in the setting
of the Cauchy-Lipschitz-Picard Theorem.

\begin{lemma}\label{PAMADS_lemma:K_tJ_t}
	 Assume that (\textit{Cweak}) holds. Let $t \in [0, +\infty)$. Then the operator
	\begin{equation*}
	K_t: \HH \to \HH, \quad K_t(u)=\argmin_{x \in \HH}\left(F(t,x)+\frac{1}{2}\|x-u\|^2\right)
	\end{equation*}
	is $\frac{1}{\sigma}$-Lipschitz continuous and the operator
	\begin{equation*}
	J_t: \GG \to \GG, \quad J_t(v)=\argmin_{z \in \GG}\left(G(t,z)+\frac{c(t)}{2}\|z-v\|^2\right)
	\end{equation*}
	is $\frac{c(t)}{\beta(t)}$-Lipschitz continuous.
\end{lemma}

\begin{proof}
Let $t \in [0, +\infty)$ be fixed. Then we have
\[0 \in \partial\left(f(\cdot)+\frac{1}{2}(\|\cdot\|^2_{M_1(t)}-\|\cdot\|^2)+\frac{1}{2}\|\cdot-u\|^2\right)(K_t(u)).\]
For all $u,v \in \HH$ we obtain
	\[u \in \partial f (K_t(u))+M_1(t)(K_t(u))\]
and \[v \in \partial f (K_t(v))+M_1(t)(K_t(v)).\]
Due to the $\sigma$-strong convexity of $f$ and $M_1(t) \in S_+(\HH)$, it follows that $\partial f+M_1(t)$ is $\sigma$-strongly monotone and we get
\[\sigma\|K_tu-K_tv\|^2\leq \langle u-v, K_t(u)-K_t(v) \rangle.\]
Using the Cauchy-Schwarz inequality it follows
\[\|K_tu-K_tv\| \leq \frac{1}{\sigma}\|u-v\|,\]
which means that $K_t$ is $\frac{1}{\sigma}$-Lipschitz continuous.

For $t \in [0, +\infty)$ fixed we have
\[0 \in \partial\left(g(\cdot)+\frac{c(t)}{2}(\|B\cdot\|^2-\|\cdot\|^2)+\frac{1}{2}\|\cdot\|^2_{M_2(t)}+\frac{c(t)}{2}\|\cdot-u\|^2\right)(J_t(u)).\]	
For all $u,v \in \GG$ we obtain
\[c(t)u \in \partial g (J_t(u))+(c(t)B^*B+M_2(t))(J_t(u))\]
and \[c(t)v \in \partial g (J_t(v))+(c(t)B^*B+M_2(t))(J_t(v)).\]
Because of (\textit{Cweak}), we have that $\partial g+c(t)B^*B+M_2(t)$ is $\beta(t)$-strongly monotone and we get
\[\beta(t)\|J_tu-J_tv\|^2\leq c(t)\langle u-v, J_t(u)-J_t(v) \rangle.\]
Using the Cauchy-Schwarz inequality it follows
\[\|J_tu-J_tv\| \leq \frac{c(t)}{\beta(t)}\|u-v\|,\]
which means that $J_t$ is $\frac{c(t)}{\beta(t)}$-Lipschitz continuous.
\end{proof}	

\begin{lemma}\label{PAMADS_lemma:RQ}
	Assume that (\textit{Cweak}) holds. Let be $(x,z,y) \in \HH \times \GG \times \mathcal{K}$ and the maps\\ $R_{(x,z,y)}:[0, +\infty) \to \HH$,
	\[R_{(x,z,y)}(t)=\argmin_{u \in \HH}\left\{F(t,u)
	+\frac{1}{2}\|u-(M_1(t)x+A^*y-\nabla h_1(x))\|^2\right\}-x,\]
	$Q_{(x,z,y)}:[0, +\infty) \to \GG$,
	\begin{align*}
	Q_{(x,z,y)}(t)=\argmin_{v \in \GG}&\left\{G(t,v)
	+\frac{c(t)}{2}\left\|v-\left(\frac{1}{c(t)}M_2(t)z+\frac{1}{c(t)}B^*y
	-B^*A(R_{(x,z,y)}(t)+x)\right.\right.\right.\\
	&\left.\left.\left.+B^*b-\frac{1}{c(t)}\nabla h_2(z))\right)\right\|^2\right\}-z,
	\end{align*}
	and $P_{(x,z,y)}:[0, +\infty) \to \mathcal{K}$,
	\begin{equation*}
	P_{(x,z,y)}(t)=c(t)(b-A(R_{(x,z,y)}(t)+x)-B(Q_{(x,z,y)}(t)+z)).
	\end{equation*}
	Then the following holds for every $t,r \in [0,+\infty):$
	\begin{align*}
		\text{(i)}&&\|R_{(x,z,y)}(t)-R_{(x,z,y)}(r)\| \leq & \frac{\|R_{(x,z,y)}(r)\|}{\sigma}\|M_1(t)-M_1(r)\|\\
		\text{(ii)}&&\|Q_{(x,z,y)}(t)-Q_{(x,z,y)}(r)\| \leq & \frac{c(t)\|A\|\|B\|\|R_{(x,z,y)}(r)\|}{\sigma\beta(t)}\|M_1(t)-M_1(r)\|\\
		&&&+\frac{\|Q_{(x,z,y)}(r)\|}{\beta(t)}\|M_2(t)-M_2(r)\|+\frac{\|P_{(x,z,y)}(r)\|\cdot\|B\|}{\beta(t)c(r)}|c(t)-c(r)|.
	\end{align*}
\end{lemma}
\begin{proof}
	Let $t,r \in [0,+\infty)$ be fixed.
	
		(i) From the definition of $R_{(x,z,y)}$ we have
		\begin{equation}\label{PAMADS:eq_lemma_RQ_partial1}
		M_1(t)x+A^*y-\nabla h_1(x) \in
		 \partial f(R_{(x,z,y)}(t)+x)+M_1(t)(R_{(x,z,y)}(t)+x)
		\end{equation}
		and
		\[M_1(r)x+A^*y-\nabla h_1(x) \in \partial f(R_{(x,z,y)}(r)+x)+M_1(r)(R_{(x,z,y)}(r)+x).\]
		If we add $M_1(t)(R_{(x,z,y)}(r)+x)$ on both sides of the relation above, we obtain
		\begin{align}
		M_1(t)(R_{(x,z,y)}(r)+x)-M_1(r)(R_{(x,z,y)}(r))+A^*y-\nabla h_1(x) \in \nonumber\\
		\partial f(R_{(x,z,y)}(r)+x)+M_1(t)(R_{(x,z,y)}(r)+x).\label{PAMADS:eq_lemma_RQ_partial2}
		\end{align}
		From \eqref{PAMADS:eq_lemma_RQ_partial1} and \eqref{PAMADS:eq_lemma_RQ_partial2} and using that $\partial f + M_1(t)$ is $\sigma$-strongly monotone, we have
		\[
		\langle M_1(t)R_{(x,z,y)}(r)-M_1(r)R_{(x,z,y)}(r),R_{(x,z,y)}(r)-R_{(x,z,y)}(t)\rangle \geq \sigma\|R_{(x,z,y)}(r)-R_{(x,z,y)}(t)\|^2.
		\]
		The result follows from the Cauchy-Schwarz inequality.
	
		(ii) From the definition of $Q_{(x,z,y)}$ we have
		\begin{align}
		M_2(t)z+B^*y-c(t)B^*A(R_{(x,z,y)}(t)+x)+c(t)B^*b-\nabla h_2(z) \in \nonumber\\
		\partial g(Q_{(x,z,y)}(t)+z)+(c(t)B^*B+M_2(t))(Q_{(x,z,y)}(t)+z)\label{PAMADS:eq_lemma_RQ_partial3}
		\end{align}
		and
		\begin{align*}
		M_2(r)z+B^*y-c(r)B^*A(R_{(x,z,y)}(r)+x)+c(r)B^*b-\nabla h_2(z) \in\\
		 \partial g(Q_{(x,z,y)}(r)+z)+(c(r)B^*B+M_2(r))(Q_{(x,z,y)}(r)+z).
		\end{align*}
		If we add $(c(t)-c(r))B^*B(Q_{(x,z,y)}(r)+z)+M_2(t)(Q_{(x,z,y)}(r)+z)$ on both sides of the relation above, we obtain
		\begin{align}
		(c(t)-c(r))B^*B(Q_{(x,z,y)}(r)+z)-M_2(r)Q_{(x,z,y)}(r)+M_2(t)(Q_{(x,z,y)}(r)+z)+B^*y&\nonumber\\
		-c(r)B^*A(R_{(x,z,y)}(r)+x)+c(r)B^*b-\nabla h_2(z) &\in \nonumber\\
		\partial g(Q_{(x,z,y)}(r)+z)+(c(t)B^*B+M_2(t))(Q_{(x,z,y)}(r)+z).&\label{PAMADS:eq_lemma_RQ_partial4}
		\end{align}
		From \eqref{PAMADS:eq_lemma_RQ_partial3} and \eqref{PAMADS:eq_lemma_RQ_partial4} and using that $\partial g + c(t)B^*B + M_2(t)$ is $\beta(t)$-strongly monotone, we have
		\begin{align*}
		\langle (c(t)-c(r))B^*B(Q_{(x,z,y)}(r)+z)+(M_2(t)-M_2(r))Q_{(x,z,y)}(r)+c(t)B^*AR_{(x,z,y)}(t)\\
		-c(r)B^*AR_{(x,z,y)}(r)+(c(t)-c(r))B^*Ax-(c(t)-c(r))B^*b,Q_{(x,z,y)}(r)-Q_{(x,z,y)}(t)\rangle \\
		\geq \beta(t)\|Q_{(x,z,y)}(r)-Q_{(x,z,y)}(t)\|^2.
		\end{align*}
		From the Cauchy-Schwarz inequality, the definition of $P_{(x,z,y)}$ and (i) it follows
		\begin{align*}
		\|Q_{(x,z,y)}(r)&-Q_{(x,z,y)}(t)\|\\&\leq \frac{1}{\beta(t)}\|(c(t)-c(r))B^*B(Q_{(x,z,y)}(r)+z)+(M_2(t)-M_2(r))Q_{(x,z,y)}(r)\\
		&\quad+c(t)B^*AR_{(x,z,y)}(t)-c(r)B^*AR_{(x,z,y)}(r)+(c(t)-c(r))B^*Ax-(c(t)-c(r))B^*b\|\\
		&=\frac{1}{\beta(t)}\|(c(t)-c(r))B^*BQ_{(x,z,y)}(r)+c(t)B^*AR_{(x,z,y)}(t)-c(r)B^*AR_{(x,z,y)}(r)\\
		&\quad+(c(t)-c(r))B^*(Ax+Bz-b)+(M_2(t)-M_2(r))Q_{(x,z,y)}(r)\|\\
		&=\frac{1}{\beta(t)}\|(c(t)-c(r))B^*BQ_{(x,z,y)}(r)+c(t)B^*AR_{(x,z,y)}(t)-c(r)B^*AR_{(x,z,y)}(r)\\
		&\quad+(c(t)-c(r))B^*(-BQ_{(x,z,y)}(r)-AR_{(x,z,y)}(r)-\frac{1}{c(r)}P_{(x,z,y)}(r))\\
		&\quad+(M_2(t)-M_2(r))Q_{(x,z,y)}(r)\|\\
		&=\frac{1}{\beta(t)}\left\|c(t)B^*AR_{(x,z,y)}(t)-c(t)B^*AR_{(x,z,y)}(r)-\frac{(c(t)-c(r))}{c(r)}B^*P_{(x,z,y)}(r)\right.\\
		&\quad\left.+(M_2(t)-M_2(r))Q_{(x,z,y)}(r)\right\|\\
		&\leq\frac{1}{\beta(t)}\left( c(t)\|A\|\|B\|\|R_{(x,z,y)}(t)-R_{(x,z,y)}(r)\|+\frac{|c(t)-c(r)|}{c(r)}\|B\|\|P_{(x,z,y)}(r)\|\right.\\
		&\left.\quad+\|M_2(t)-M_2(r)\|\|Q_{(x,z,y)}(r)\|\right)\\
		&\leq\frac{1}{\beta(t)}\left( \frac{c(t)\|A\|\|B\|\|R_{(x,z,y)}(r)\|}{\sigma}\|M_1(t)-M_1(r)\|+\frac{|c(t)-c(r)|}{c(r)}\|B\|\|P_{(x,z,y)}(r)\|\right.\\
		&\left.\quad+\|M_2(t)-M_2(r)\|\|Q_{(x,z,y)}(r)\|\right).\\
		\end{align*}
\end{proof}

Having now all these estimations at our disposal, we are now ready to prove the existence and
uniqueness of the trajectories.

\begin{theorem}
	Assume that (\textit{Cstrong}) holds, $M_1 \in L_{loc}^1([0,+\infty),\HH)$ and $M_2 \in L_{loc}^1([0,+\infty),\GG)$. Furthermore we assume that $0<\inf_{t\geq 0} c(t)\leq\sup_{t\geq 0}c(t)<+\infty$. Then, for every starting points $(x^0,z^0,y^0) \to \HH \times \GG \times \mathcal{K}$, the dynamical system \eqref{eq:proxAMA-DS} has a unique strong global solution $(x,z,y):[0,+\infty) \to\HH \times \GG \times \mathcal{K}$.
\end{theorem}

\begin{proof}
	In the following we use the equivalent formulation of the dynamical system described in Remark \ref{re:proxAMA-DS_uvw}. We show the existence and uniqueness of a strong global solution using the Cauchy-Lipschitz-Picard Theorem. To this end, we rely on \cite[Proposition 6.2.1]{haraux} and \cite{att-sv2011} (see Theorem 2.4, ODE (37) and the conditions (42), (44) and (45)). 
	
	In the first part we have to show, that $\Gamma(t,\cdot,\cdot,\cdot)$ is $L(t)$-Lipschitz continuous for every $t \in [0,+\infty)$ and that the Lipschitz constant as a function of time fulfills $L(\cdot)\in L_{loc}^1([0,+\infty),\R)$. In the second part we will prove that $\Gamma(\cdot,x,z,y)\in L_{loc}^1([0,+\infty),\HH \times \GG \times \mathcal{K})$ for every $(x,z,y) \in \HH \times \GG \times \mathcal{K}$.
	
	(1) Let $t\in[0,+\infty)$ be fixed and let $(x,z,y),(\overline{x},\overline{z},\overline{y})\in\HH \times \GG \times \mathcal{K})$. We have
	\[
	\|\Gamma(t,x,z,y)-\Gamma(t,\overline{x},\overline{z},\overline{y})\|=\sqrt{\|u-\overline{u}\|^2+\|v-\overline{v}\|^2+\|w-\overline{w}\|^2}
	\]
	where (taking into account Lemma \ref{PAMADS_lemma:K_tJ_t})
	\begin{align*}
	u-\overline{u}&=\argmin_{p \in \HH}\left\{F(t,p)
	+\frac{1}{2}\|p-(M_1(t)x+A^*y-\nabla h_1(x))\|^2\right\}\\
	&\quad-\argmin_{p \in \HH}\left\{F(t,p)
	+\frac{1}{2}\|p-(M_1(t)\overline{x}+A^*\overline{y}-\nabla h_1(\overline{x}))\|^2\right\}+\overline{x}-x\\
	&=K_t(M_1(t)x+A^*y-\nabla h_1(x))-K_t(M_1(t)\overline{x}+A^*\overline{y}-\nabla h_1(\overline{x}))+\overline{x}-x.
	\end{align*}
	Therefore,
	\begin{align*}
	\|u-\overline{u}\|^2 \leq 2 \|K_t(M_1(t)x+A^*y-\nabla h_1(x))-K_t(M_1(t)\overline{x}+A^*\overline{y}-\nabla h_1(\overline{x}))\|^2+2\|\overline{x}-x\|^2.
	\end{align*}
	From Lemma \ref{PAMADS_lemma:K_tJ_t}~ we know, that $K_t$ is $\frac{1}{\sigma}$-Lipschitz-continuous. Thus:
	\begin{align*}
	\|u-\overline{u}\|^2 &\leq \frac{2}{\sigma^2}\|M_1(t)(x-\overline{x})+A^*(y-\overline{y})-(\nabla h_1(x)-\nabla h_1(\overline{x}))\|^2+2\|\overline{x}-x\|^2\\
	&\leq\frac{2}{\sigma^2}(2\|M_1(t)(x-\overline{x})+A^*(y-\overline{y})\|^2+2\|\nabla h_1(x)-\nabla h_1(\overline{x})\|^2)+2\|\overline{x}-x\|^2\\
	&\leq\frac{2}{\sigma^2}(4\|M_1(t)\|^2\|x-\overline{x}\|^2+4\|A\|^2\|y-\overline{y}\|^2+2\|\nabla h_1(x)-\nabla h_1(\overline{x})\|^2)+2\|\overline{x}-x\|^2\\
	&\leq 2\left(\frac{4\|M_1(t)\|^2+2L_{h_1}^2}{\sigma^2}+1\right)\|x-\overline{x}\|^2+\frac{8\|A\|^2}{\sigma^2}\|y-\overline{y}\|^2.
	\end{align*}
	Furthermore by taking into account Lemma \ref{PAMADS_lemma:K_tJ_t}~ we have
	\begin{align*}
	v-\overline{v}&=\argmin_{q \in \GG}\left\{G(t,q)
	+\frac{c(t)}{2}\left\|q-\left(\frac{1}{c(t)}M_2(t)z+\frac{1}{c(t)}B^*y
	-B^*A(u+x)+B^*b-\frac{1}{c(t)}\nabla h_2(z))\right)\right\|^2\right\}\\
	&\quad -\argmin_{q \in \GG}\left\{G(t,q)+
	\frac{c(t)}{2}\left\|q-\left(\frac{1}{c(t)}M_2(t)\overline{z}+\frac{1}{c(t)}B^*\overline{y}
	-B^*A(\overline{u}+\overline{x})+B^*b-\frac{1}{c(t)}\nabla h_2(\overline{z}))\right)\right\|^2\right\}\\
	&\quad+\overline{z}-z\\
	&=J_t\left(\frac{1}{c(t)}M_2(t)z+\frac{1}{c(t)}B^*y
	-B^*A(u+x)+B^*b-\frac{1}{c(t)}\nabla h_2(z))\right)\\
	&\quad -J_t\left(\frac{1}{c(t)}M_2(t)\overline{z}+\frac{1}{c(t)}B^*\overline{y}
	-B^*A(\overline{u}+\overline{x})+B^*b-\frac{1}{c(t)}\nabla h_2(\overline{z}))\right)+\overline{z}-z.
	\end{align*}
	According to Lemma \ref{PAMADS_lemma:K_tJ_t}~ and (\textit{Cstrong}) we have that $J_t$ is $\frac{c(t)}{\beta}$-Lipschitz-continuous. We derive:
	\begin{align*}
	\|v-\overline{v}\|^2&\leq 2\left\|J_t\left(\frac{1}{c(t)}M_2(t)z+\frac{1}{c(t)}B^*y
	-B^*A(u+x)+B^*b-\frac{1}{c(t)}\nabla h_2(z))\right)-\right
	.\\
	&\quad \left. J_t\left(\frac{1}{c(t)}M_2(t)\overline{z}+\frac{1}{c(t)}B^*\overline{y}
	-B^*A(\overline{u}+\overline{x})+B^*b-\frac{1}{c(t)}\nabla h_2(\overline{z}))\right)\right\|^2+2\|\overline{z}-z\|^2\\
	&\leq \frac{2c^2(t)}{\beta^2}\left\|\frac{1}{c(t)}M_2(t)(z-\overline{z})+\frac{1}{c(t)}B^*(y-\overline{y})
	-B^*A(u-\overline{u}+x-\overline{x})-\frac{1}{c(t)}(\nabla h_2(z)-\nabla h_2(\overline{z}))\right\|^2\\
	&\quad+2\|z-\overline{z}\|^2\\
	&\leq \frac{2c^2(t)}{\beta^2}\left(\frac{4}{c^2(t)}\|M_2(t)\|^2\|z-\overline{z}\|^2+\frac{4}{c^2(t)}\|B\|^2\|y-\overline{y}\|^2
	+4\|B\|^2\|A\|^2\|u-\overline{u}+x-\overline{x}\|^2\right.\\
	&\left.\quad+\frac{4}{c^2(t)}\|\nabla h_2(z)-\nabla h_2(\overline{z})\|^2\right)+2\|z-\overline{z}\|^2\\
	&\leq \frac{8}{\beta^2}\|M_2(t)\|^2\|z-\overline{z}\|^2+\frac{8}{\beta^2}\|B\|^2\|y-\overline{y}\|^2
	+\frac{16c^2(t)}{\beta^2}\|B\|^2\|A\|^2\|u-\overline{u}\|^2\\
	&\quad+\frac{16c^2(t)}{\beta^2}\|B\|^2\|A\|^2\|x-\overline{x}\|^2+\frac{8}{\beta^2}\|\nabla h_2(z)-\nabla h_2(\overline{z})\|^2+2\|z-\overline{z}\|^2\\
	&\leq \left(\frac{8\|M_2(t)\|^2+8L_{h_2}^2}{\beta^2}+2\right)\|z-\overline{z}\|^2+\frac{8}{\beta^2}\|B\|^2\|y-\overline{y}\|^2
	+\frac{16c^2(t)}{\beta^2}\|B\|^2\|A\|^2\|u-\overline{u}\|^2\\
	&\quad+\frac{16c^2(t)}{\beta^2}\|B\|^2\|A\|^2\|x-\overline{x}\|^2\\
	&\leq \left(\frac{8\|M_2(t)\|^2+8L_{h_2}^2}{\beta^2}+2\right)\|z-\overline{z}\|^2+\frac{8}{\beta^2}\|B\|^2\|y-\overline{y}\|^2\\
	&\quad+\frac{16c^2(t)}{\beta^2}\|B\|^2\|A\|^2\left(\left(\frac{8\|M_1(t)\|^2+4L_{h_1}^2}{\sigma^2}+3\right)\|x-\overline{x}\|^2+\frac{8\|A\|^2}{\sigma^2}\|y-\overline{y}\|^2\right)\\
	&=\frac{16c^2(t)}{\beta^2}\|A\|^2\|B\|^2\left(\frac{8\|M_1(t)\|^2+4L_{h_1}^2}{\sigma^2}+3\right)\|x-\overline{x}\|^2+\frac{8}{\beta^2}\|B\|^2\left(1+\frac{16c^2(t)}{\sigma^2}\|A\|^4\right)\|y-\overline{y}\|^2\\
	&\quad + \left(\frac{8\|M_2(t)\|^2+8L_{h_2}^2}{\beta^2}+2\right)\|z-\overline{z}\|^2.
	\end{align*}
	Finally
	\begin{align*}
	\|w-\overline{w}\|^2&=\|-c(t)(A(u-\overline{u}+x-\overline{x})+B(v-\overline{v}+z-\overline{z}))\|^2\\
	&\leq 4c^2(t)\|A\|^2\|u-\overline{u}\|^2+4c^2(t)\|A\|^2\|x-\overline{x}\|^2+4c^2(t)\|B\|^2\|v-\overline{v}\|^2+4c^2(t)\|B\|^2\|z-\overline{z}\|^2\\
	&\leq 4c^2(t)\|A\|^2\left(3+\frac{8\|M_1(t)\|^2+4L_{h_1}^2}{\sigma^2}+\frac{16c^2(t)}{\beta^2}\|B\|^4\left(\frac{8\|M_1(t)\|^2+4L_{h_1}^2}{\sigma^2}+3\right)\right)\|x-\overline{x}\|^2\\
	&\quad+4c^2(t)\|B\|^2 \left(\frac{8\|M_2(t)\|^2+8L_{h_2}^2}{\beta^2}+3\right)\|z-\overline{z}\|^2\\
	&\quad+32c^2(t)\left(\frac{\|B\|^4}{\beta^2}+\frac{16c^2(t)}{\sigma^2\beta^2}\|A\|^4\|B\|^4+\frac{\|A\|^4}{\sigma^2}\right)\|y-\overline{y}\|^2.
	\end{align*}
	Then we have,
	\begin{align*}
	\|\Gamma(t,x,z,y)-\Gamma(t,\overline{x},\overline{z},\overline{y})\|&\leq\sqrt{L_1(t)\|x-\overline{x}\|^2+L_2(t)\|z-\overline{z}\|^2+L_3(t)\|w-\overline{w}\|^2}\\
	&\leq\sqrt{L_1(t)+L_2(t)+L_3(t)}\sqrt{\|x-\overline{x}\|^2+\|z-\overline{z}\|^2+\|w-\overline{w}\|^2}\\
	&=L(t)\|(x,z,y)-(\overline{x},\overline{z},\overline{y})\|,
	\end{align*}
	where
	\[L(t)=\sqrt{L_1(t)+L_2(t)+L_3(t)}\]
	and
	\begin{align*}
	L_1(t)=&\left(2+\frac{32}{\beta^2}\|A\|^2\|B\|^2+8c^2(t)\|A\|^2\left(1+\frac{16c^2(t)}{\beta^2}\|B\|^4\right)\right)\left(\frac{4\|M_1(t)\|^2+2L_{h_1}^2}{\sigma^2}+1\right)\\
	L_2(t)=&2+\frac{8\|M_2(t)\|^2+8L_{h_2}^2}{\beta^2}+4c^2(t)\|B\|^2 \left(\frac{8\|M_2(t)\|^2+8L_{h_2}^2}{\beta^2}+3\right)\\
	L_3(t)=&\frac{8\|A\|^2}{\sigma^2}+\frac{8}{\beta^2}\|B\|^2\left(1+\frac{16c^2(t)}{\sigma^2}\|A\|^4\right)+32c^2(t)\left(\frac{\|B\|^4}{\beta^2}+\frac{16c^2(t)}{\sigma^2\beta^2}\|A\|^4\|B\|^4+\frac{\|A\|^4}{\sigma^2}\right),
	\end{align*}
	which means that $\Gamma(t,\cdot,\cdot,\cdot)$ is $L(t)$-Lipschitz continuous. Since $M_1 \in L_{loc}^1([0,+\infty),\HH)$, $M_2 \in L_{loc}^1([0,+\infty),\GG)$ and $c(t)$ is bounded, it follows that $L(\cdot)\in L_{loc}^1([0,+\infty),\R)$.
	
	(2)~Now we will show that $\Gamma(\cdot,x,z,y) \in  L_{loc}^1([0,+\infty),\HH\times \GG \times \mathcal{K})$ for every $(x,z,y) \in \HH\times \GG \times \mathcal{K}$.
	Let $(x,z,y) \in \HH\times \GG \times \mathcal{K}$ be fixed and $T>0$. We have
	\[\int_{0}^{T}\|\Gamma(t,x,z,y)\|dt=\int_{0}^{T}\sqrt{\|u(t,x,z,y)\|^2+\|v(t,x,z,y)\|^2+\|w(t,x,z,y)\|^2}dt.\]
	From Lemma \ref{PAMADS_lemma:RQ}~ and the fact that $\sigma>0$ and $\beta(t)=\beta>0$, for all $t\in0,+\infty)$, we have
	\begin{align*}
	\|u(t,x,z,y)\|^2 & \leq 2\|u(t,x,z,y)-u(0,x,z,y)\|^2+2\|u(0,x,z,y)\|^2\\
	& \leq \frac{2\|u(0,x,z,y)\|^2}{\sigma^2}\|M_1(t)-M_1(0)\|^2+2\|u(0,x,z,y)
	\|^2,
	\end{align*}
	\begin{align*}
	\|v(t,x,z,y)\|^2 & \leq 2\|v(t,x,z,y)-v(0,x,z,y)\|^2+2\|v(0,x,z,y)\|^2\\
	& \leq \frac{6c^2(t)\|A\|^2\|B\|^2\|u(0,x,z,y)\|^2}{\sigma^2\beta^2}\|M_1(t)-M_1(0)\|^2\\
	& \quad+\frac{6\|v(0,x,z,y)\|^2}{\beta^2}\|M_2(t)-M_2(0)\|^2+\frac{6\|w(0,x,z,y)\|^2\cdot\|B\|^2}{c^2_0\beta^2}|c(t)-c_0|^2+2\|v(0,x,z,y)
	\|^2,
	\end{align*}
	(where $c_0=c(0)$) and
	\begin{align*}
	\|w(t,x,z,y)\|^2 & \leq c^2(t)\|b-A(u(t,x,z,y)+x)-B(v(t,x,z,y)+z))\|^2\\
	& \leq 4c^2(t)(\|b\|^2+\|A\|^2\|u(t,x,z,y)\|^2+\|B\|^2\|v(t,x,z,y)\|^2+\|Ax-Bz\|^2)\\
	& \leq 4c^2(t)\left(\|b\|^2+\frac{2\|A\|^2\|u(0,x,z,y)\|^2}{\sigma^2}\|M_1(t)-M_1(0)\|^2\right.\\
	&\left.\quad+2\|A\|^2\|u(0,x,z,y)\|^2+\frac{6c^2(t)\|A\|^2\|B\|^4\|u(0,x,z,y)\|^2}{\sigma^2\beta^2}\|M_1(t)-M_1(0)\|^2\right.\\
	&\left.\quad+\frac{6\|B\|^2\|v(0,x,z,y)\|^2}{\beta^2}\|M_2(t)-M_2(0)\|^2+\frac{6\|B\|^4\|w(0,x,z,y)\|^2}{c^2_0\beta^2}|c(t)-c_0|^2\right.\\
	&\left. \quad+2\|B\|^2\|v(0,x,z,y)\|^2+\|Ax-Bz\|^2\right).
	\end{align*}
	Since $M_1 \in L_{loc}^1([0,+\infty),\HH)$, $M_2 \in L_{loc}^1([0,+\infty),\GG)$ and $c(t)$ is bounded, it follows that the integral
	\[\int_{0}^{T}\|\Gamma(t,x,z,y)\|dt\]
	exists and it is finite. So we have that $\Gamma(\cdot,x,z,y) \in  L_{loc}^1([0,+\infty),\HH\times \GG \times \mathcal{K})$.
	The conclusion follows.
\end{proof}

\section{Convergence of the trajectories}

In the beginning of this section we will give some results, which we will use then to prove the convergence of the trajectories of the dynamical system \eqref{eq:proxAMA-DS}. In the following the real vector space $\mathcal{L}(\HH):=\{A:\HH \to \HH: A \text{ is linear and continuous}\}$ is endowed 
with the norm $\|A\|=\sup_{\|x\|\leq 1}\|Ax\|.$

\begin{definition}\label{def-mdot}
	The map $M:[0,+\infty) \to \mathcal{L}(\HH) $ is said to be derivable at $t_0\in [0,+\infty)$, if the limit
	\begin{equation}\label{mdot}
	\lim\limits_{h\to 0}\frac{M(t_0+h)-M(t_0)}{h}
	\end{equation}
	taken with respect to the norm topology of $\mathcal{L}(\HH)$ exists. When this is the case, we denote by $\dot{M}(t_0)\in \mathcal{L}(\HH)$ the value of the limit.
\end{definition}

In case $M:[0,+\infty)\to\mathcal{L}(\HH)$ is derivable at $t_0 \in [0,+\infty)$ and $x,y:[0,+\infty)\to \HH$ are also derivable at $t_0$, we will use the following formula (see \cite[Lemma 4]{bocs19}): 
	\begin{equation}\label{diff}\frac{d}{dt}\langle M(t)x(t),y(t)\rangle |_{t=t_0}=\langle \dot{M}(t_0)x(t_0),y(t_0)\rangle+\langle M(t_0)\dot{x}(t_0),y(t_0)\rangle+\langle M(t_0)x(t_0),\dot{y}(t_0)\rangle.\end{equation}

We start with a result where we show that under appropriate conditions the second derivatives of the trajectories exist
almost everywhere and give also an upper bound on their norms. This will be used in the proof of the main result
Theorem \ref{PAMADS:th:convergence}.

\begin{lemma}\label{PAMADS_lemma:xz_dot}
	Assume that (\textit{Cstrong}) holds and that the maps $M_1:[0,+\infty)\to S_+(\HH)$ and $M_2:[0,+\infty)\to S_+(\GG)$ are locally absolutely continuous. Furthermore we assume that $c$ is locally absolutely continuous and
$0<\inf_{t\geq 0} c(t)\leq\sup_{t\geq 0}c(t)<+\infty$. For a given starting point $(x^0,z^0,y^0) \in \HH \times \GG \times \mathcal{K}$ let $(x,z,y):[0,+\infty)\to\HH \times \GG \times \mathcal{K}$ be the unique strong  global solution of the dynamical system \eqref{eq:proxAMA-DS}. Then
	\[t\to(\dot{x}(t),\dot{z}(t),\dot{y}(t))\]
	is locally absolutely continuous, hence $(\ddot{x}(t),\ddot{z}(t),\ddot{y}(t))$ exists for almost every $t \in [0,+\infty)$.
	In addition, if\\ $\sup_{t\geq 0}\|M_1(t)\|<+\infty$ and $\sup_{t\geq 0}\|M_2(t)\|<+\infty$, then there exists $L>0$ such that
	\[
	\|\ddot{x}(t)\|+\|\ddot{z}(t)\|\leq L(\|\dot{x}(t)\|+\|\dot{z}(t)\|+\|\dot{y}(t)\|+\|\dot{M_1}(t)\|\|\dot{x}(t)\|+\|\dot{M_2}(t)\|\|\dot{z}(t)\|+|\dot{c}(t)|\|\dot{y}(t)\|)
	\]
	for almost every $t\in [0,+\infty)$.
\end{lemma}

\begin{proof}
	Let $T>0$ be fixed. In the following we use the notation \eqref{eq:proxAMA-DS_uvw} again. Let $t,r \in [0,T]$ be fixed. We have
	\begin{align*}
	\|\dot{U}(t)-&\dot{U}(r)\|=\|\Gamma(t,U(t))-\Gamma(r,U(r))\|\leq\|\Gamma(t,U(t))-\Gamma(t,U(r))\|+\|\Gamma(t,U(r))+\Gamma(r,U(r))\|\\
	\leq&\|u(t,x(t),z(t),y(t))-u(t,x(r),z(r),y(r))\|+\|v(t,x(t),z(t),y(t))-v(t,x(r),z(r),y(r))\|\\
	&+\|w(t,x(t),z(t),y(t))-w(t,x(r),z(r),y(r))\|+\|u(t,x(r),z(r),y(r)-u(r,x(r),z(r),y(r))\|\\
	&+\|v(t,x(r),z(r),y(r))-v(r,x(r),z(r),y(r))\|+\|w(t,x(r),z(r),y(r))-w(r,x(r),z(r),y(r))\|.
	\end{align*}
	Since
	\begin{align*}
	u(t,x(t),z(t),y(t))-u(t,x(r),z(r),y(r))=&K_t(M_1(t)x(t)+A^*y(t)-\nabla h_1(x(t)))\\
	& -K_t(M_1(t)x(r)+A^*y(r)-\nabla h_1(x(r)))-x(t)+x(r)
	\end{align*}
	due to Lemma \ref{PAMADS_lemma:K_tJ_t}~ we get
	\begin{align*}
	\|u&(t,x(t),z(t),y(t))-u(t,x(r),z(r),y(r))\|\\
	&\leq\frac{1}{\sigma}\|M_1(t)(x(t)-x(r))+A^*(y(t)-y(r))
	-(\nabla h_1(x(t))-\nabla h_1(x(r)))\|+\|x(t)-x(r)\|\\
	&\leq\left(\frac{\|M_1(t)\|}{\sigma}+\frac{L_{h_1}}{\sigma}+1\right)\|x(t)-x(r)\|+\frac{\|A\|}{\sigma}\|y(t)-y(r)\|.
	\end{align*}
	Since $t \to \|M_1(t)\|$ is bounded on $[0,T]$, there exists $L_1:=L_1(T)>0$ such that
	\begin{equation}\label{PAMADS_eq:norm_u1}
	\|u(t,x(t),z(t),y(t))-u(t,x(r),z(r),y(r))\|\leq L_1(\|x(t)-x(r)\|+\|y(t)-y(r)\|).
	\end{equation}
	Similarly we have
	\begin{align*}
	v(t,&x(t),z(t),y(t))-v(t,x(r),z(r),y(r))\\
	=&J_t\left(\frac{1}{c(t)}M_2(t)z(t)+\frac{1}{c(t)}B^*y(t)-B^*A(u(t,x(t),z(t),y(t))+x(t))+B^*b-\frac{1}{c(t)}\nabla h_2(z(t))\right)\\
	&-J_t\left(\frac{1}{c(t)}M_2(t)z(r)+\frac{1}{c(t)}B^*y(r)-B^*A(u(t,x(r),z(r),y(r))+x(r))+B^*b-\frac{1}{c(t)}\nabla h_2(z(r))\right)\\
	&-z(t)+z(r),
	\end{align*}
	and again according to Lemma \ref{PAMADS_lemma:K_tJ_t}~, we get
	\begin{align*}
	\|v&(t,x(t),z(t),y(t))-v(t,x(r),z(r),y(r))\|\\
	&\leq \frac{c(t)}{\beta}\left\|\frac{1}{c(t)}M_2(t)(z(t)-z(r))+\frac{1}{c(t)}B^*(y(t)-y(r))-B^*A(u(t,x(t),z(t),y(t))\right.\\
	&\quad\left.-u(t,x(r),z(r),y(r))+x(t)-x(r))-\frac{1}{c(t)}(\nabla h_2(z(t))-\nabla h_2(z(r)))\right\|+\|z(t)-z(r)\|\\
	&\leq\frac{c(t)}{\beta}\|A\|\|B\|(L_1+1)\|x(t)-x(r)\|+\left(\frac{\|M_2(t)\|}{\beta}+\frac{L_{h_2}}{\beta}+1\right)\|z(t)-z(r)\|\\
	&\quad+\frac{\|B\|}{\beta}(1+c(t)\|A\|L_1)\|y(t)-y(r)\|.
	\end{align*}
	Since $t \to \|M_2(t)\|$ is bounded on $[0,T]$, there exists $L_2:=L_2(T)>0$ such that
	\begin{equation}\label{PAMADS_eq:norm_v1}
	\|v(t,x(t),z(t),y(t))-v(t,x(r),z(r),y(r))\|\leq L_2(\|x(t)-x(r)\|+\|z(t)-z(r)\|+\|y(t)-y(r)\|).
	\end{equation}
	Using \eqref{PAMADS_eq:norm_u1} and \eqref{PAMADS_eq:norm_v1} we obtain
	\begin{align*}
	\|w&(t,x(t),z(t),y(t))-w(t,x(r),z(r),y(r))\|\\
	&\leq c(t)\|A(u(t,x(t),z(t),y(t))-u(t,x(r),z(r),y(r))+x(t)-x(r))\|\\
	&\quad+c(t)\|B\|\|v(t,x(t),z(t),y(t))-v(t,x(r),z(r),y(r))+z(t)-z(r)\|\\
	&\leq c(t)(\|A\|L_1+\|B\|L_2+\|A\|)\|x(t)-x(r)\|+c(t)(\|B\|L_2+\|B\|)\|z(t)-z(r)\|\\
	&\quad+c(t)(\|A\|L_1+\|B\|L_2)\|y(t)-y(r)\|.
	\end{align*}
	So, there exists $L_3(T):=L_3:=\sup_{t \in [0,T]}c(t)(\|A\|L_1+\|B\|L_2+\|A\|+\|B\|)>0$ such that
	\begin{equation}\label{PAMADS_eq:norm_w1}
	\|w(t,x(t),z(t),y(t))-w(t,x(r),z(r),y(r))\|\leq L_3(\|x(t)-x(r)\|+\|z(t)-z(r)\|+\|y(t)-y(r)\|).
	\end{equation}
	Using Lemma \ref{PAMADS_lemma:RQ}(i), we obtain
	\begin{align}
	\|u(t,x(r),z(r),y(r))-u(r,x(r),z(r),y(r))\|&=\|R_{(x(r),z(r),y(r))}(t)-R_{(x(r),z(r),y(r))}(r)\|\nonumber\\
	&\leq\frac{\|R_{(x(r),z(r),y(r))}(r)\|}{\sigma}\|M_1(t)-M_1(r)\|.\label{PAMADS_eq:norm_u2}
	\end{align}
	Since $K_r$ and $\nabla h_1$ are Lipschitz continuous (see Lemma \ref{PAMADS_lemma:K_tJ_t}) and $x,z,y$ and $M_1$ are absolutely continuous on $[0,T]$, the map
	\[
	r \mapsto R_{(x(r),z(r),y(r))}(r) =K_r(M_1(r)x(r)+A^*y(r)-\nabla h_1(x(r)))-x(r)
	\]
	is bounded in $[0,T]$. Therefore, there exists $L_4:=L_4(T)>0$ such that
	\begin{equation*}
	\|u(t,x(r),z(r),y(r))-u(r,x(r),z(r),y(r))\|\leq L_4\|M_1(t)-M_1(r)\|.
	\end{equation*}
	In an analog way, using Lemma \ref{PAMADS_lemma:RQ}(ii), we get
	\begin{align}
	\|v(t&,x(r),z(r),y(r))-v(r,x(r),z(r),y(r))\|=\|Q_{(x(r),z(r),y(r))}(t)-Q_{(x(r),z(r),y(r))}(r)\|\nonumber\\
	&\leq\frac{c(t)\|A\|\|B\|\|R_{(x(r),z(r),y(r))}(r)\|}{\sigma\beta}\|M_1(t)-M_1(r)\|+\frac{\|Q_{(x(r),z(r),y(r))}(r)\|}{\beta}\|M_2(t)-M_2(r)\|\nonumber\\
	&\quad +\frac{\|P_{(x(r),z(r),y(r))}(r)\|\cdot\|B\|}{c(r)\beta}|c(t)-c(r)|.\label{PAMADS_eq:norm_v2}
	\end{align}
	Since $J_r$ and $\nabla h_2$ are Lipschitz continuous (see Lemma \ref{PAMADS_lemma:K_tJ_t}), $x,z,y$ and $M_2$ are absolutely continuous
on $[0,T]$ and $c$ is bounded, the maps
	\begin{align*}
	r \mapsto Q_{(x(r),z(r),y(r))}(r) =&J_r\left(\frac{M_2(r)}{c(r)}z(r)+\frac{B^*}{c(r)}y(r)-B^*A(u(r,x(r),z(r),y(r)+x(r))\right.\\
	&\left.+B^*b-\frac{1}{c(r)}\nabla h_2(z(r))\right)-z(r)
	\end{align*}
	and
	\[
	r \mapsto P_{(x(r),z(r),y(r))}(r)=c(r)\left(b-A(R_{(x(r),z(r),y(r))}(r)+x(r))-B(Q_{(x(r),z(r),y(r))}(r)+z(r)\right)
	\]
	are bounded in $[0,T]$. Therefore, there exists $L_5:=L_5(T)>0$ such that
	\begin{align*}
	\|v(t,x(r),z(r),y(r))-v(r,x(r),z(r),y(r))\|\leq L_5(&\|M_1(t)-M_1(r)\|+\|M_2(t)-M_2(r)\|\\
	&+|c(t)-c(r)|).
	\end{align*}
	Further by using \eqref{PAMADS_eq:norm_u2} and \eqref{PAMADS_eq:norm_v2}, we obtain
	\begin{align*}
	\|w&(t,x(r),z(r),y(r))-w(r,x(r),z(r),y(r))\|\\
	&\leq \|(c(t)-c(r))b-(c(t)-c(r))Ax(r)-(c(t)-c(r))Bz(r)-A(c(t)(u(t,x(r),z(r),y(r))\\
	&\quad-c(r)(u(r,x(r),z(r),y(r))))-B(c(t)(v(t,x(r),z(r),y(r))-c(r)(v(r,x(r),z(r),y(r))))\|\\	
	&\leq \|b\||c(t)-c(r)|+\|A\|\|x(r)\||c(t)-c(r)|+\|B\|\|z(r)\||c(t)-c(r)|\\
	&\quad+\|A(c(t)u(t,x(r),z(r),y(r)-c(r)u(r,x(r),z(r),y(r)))\|\\
	&\quad+\|B(c(t)v(t,x(r),z(r),y(r)-c(r)v(r,x(r),z(r),y(r)))\|\\
	&=(\|b\|+\|A\|\|x(r)\|+\|B\|\|z(r)\|)|c(t)-c(r)|\\
	&\quad+\|A(c(t)u(t,x(r),z(r),y(r)-c(r)u(t,x(r),z(r),y(r)))\|\\
	&\quad+\|A(c(r)u(t,x(r),z(r),y(r)-c(r)u(r,x(r),z(r),y(r)))\|\\
	&\quad+\|B(c(t)v(t,x(r),z(r),y(r)-c(r)v(t,x(r),z(r),y(r)))\|\\
	&\quad+\|B(c(r)v(t,x(r),z(r),y(r)-c(r)v(r,x(r),z(r),y(r)))\|\\
	&\leq(\|b\|+\|A\|\|x(r)\|+\|B\|\|z(r)\|+\|A\|\|(u(t,x(r),z(r),y(r))\|\\
	&\quad+\|B\|\|v(t,x(r),z(r),y(r))\|+\|B\|L_5)|c(t)-c(r)|+c(r)(\|A\|L_4+\|B\|L_5)\|M_1(t)-M_1(r)\|\\
	&\quad+c(r)\|B\|L_5\|M_2(t)-M_2(r)\|.
	\end{align*}
	So, there exists $L_6:=L_6(T)=\sup_{r \in [0,T]}(\|b\|+\|A\|\|x(r)\|+\|B\|\|z(r)\|+\|A\|\|(u(t,x(r),z(r),y(r))\|+\|B\|\|v(t,x(r),z(r),y(r))\|+c(r)\|A\|L_4+c(r)\|B\|L_5+\|B\|L_5)>0$ such that
	\begin{align}
	\|w(t,x(r)&,z(r),y(r))-w(r,x(r),z(r),y(r))\|\nonumber\\
	&\leq L_6(\|M_1(t)-M_1(r)\|+\|M_2(t)-M_2(r)\|+|c(t)-c(r)|).\label{PAMADS_eq:norm_w2}
	\end{align}
	
	Summing the relations \eqref{PAMADS_eq:norm_u1}-\eqref{PAMADS_eq:norm_w2} we get that there exists $L_7:=L_7(T)>0$ such that
	\begin{align*}
	\|\dot{U}(t)-\dot{U}(r)\| \leq L_7(&\|x(t)-x(r)\|+\|z(t)-z(r)\|+\|y(t)-y(r)\|\\
	&+\|M_1(t)-M_1(r)\|+\|M_2(t)-M_2(r)\|+|c(t)-c(r)|).
	\end{align*}
	Let $\epsilon>0$. Since the maps $x,z,y, M_1$, $M_2$ and $c$ are absolutely continuous on $[0,T]$, there exists $\eta>0$ such that for any finite family of intervals $I_k=(a_k,b_k) \subseteq [0,T]$ such that for any subfamily of disjoint intervals $I_j$ with $\sum_{j}|b_j-a_j|<\eta$ holds
	\[
	\sum_{j}\|x(b_j)-x(a_j)\|<\frac{\epsilon}{6L_7},\quad \sum_{j}\|z(b_j)-z(a_j)\|<\frac{\epsilon}{6L_7},\quad \sum_{j}\|y(b_j)-y(a_j)\|<\frac{\epsilon}{6L_7},
	\]
	\[
	\sum_{j}\|M_1(b_j)-M_1(a_j)\|<\frac{\epsilon}{6L_7},\quad \sum_{j}\|M_2(b_j)-M_2(a_j)\|<\frac{\epsilon}{6L_7}\text{ and }
	\sum_{j}|c(b_j)-c(a_j)|<\frac{\epsilon}{6L_7}.
	\]
	So we have
	\[
	\sum_{j}\|\dot{U}(b_j)-\dot{U}(a_j)\|<\epsilon,
	\]
	and therefore $\dot{U}(\cdot)=(\dot{x}(\cdot),\dot{z}(\cdot),\dot{y}(\cdot))$ is absolutely continuous on $[0,T]$. This proves that the second order derivatives $\ddot{x}, \ddot{z}, \ddot{y}$ exists almost everywhere on $[0,+\infty)$.
	
	To prove the second statement we assume that  $\sup_{t\geq 0}\|M_1(t)\|<+\infty$ and $\sup_{t\geq 0}\|M_2(t)\|<+\infty$. Note that $c(t)$ is bounded for all $t \in [0,+\infty)$. Then $L_1,L_2$ and $L_3$ can be taken as being global constants, so that \eqref{PAMADS_eq:norm_u1}, \eqref{PAMADS_eq:norm_v1} and \eqref{PAMADS_eq:norm_w1} hold for every $t,r \in [0,+\infty)$.
	
	Since $R_{(x(r),z(r),y(r)}(r)=\dot{x}(r)$, $Q_{(x(r),z(r),y(r)}(r)=\dot{z}(r)$ and $P_{(x(r),z(r),y(r)}(r)=\dot{y}(r)$ for every $r\in [0,+\infty)$ and taking into account \eqref{PAMADS_eq:norm_u2} and \eqref{PAMADS_eq:norm_v2} we get
	\begin{equation}\label{PAMADS_eq:norm_u2_global}
	\|u(t,x(r),z(r),y(r))-u(r,x(r),z(r),y(r))\|\leq\frac{\|\dot{x}(r)\|}{\sigma}\|M_1(t)-M_1(r)\|.
	\end{equation}
	and, respectively,
	\begin{align}
	\|v(t,x(r),z(r),y(r))-v(r,x(r),z(r),y(r))\|
	\leq&\frac{c(t)\|A\|\|B\|\|\dot{x}(r)\|}{\sigma\beta}\|M_1(t)-M_1(r)\|\nonumber\\
	&+\frac{\|\dot{z}(r)\|}{\beta}\|M_2(t)-M_2(r)\|+\frac{\|\dot{y}(r)\|\cdot\|B\|}{\beta c(r)}|c(t)-c(r)|\label{PAMADS_eq:norm_v2_global}
	\end{align}
	for every $t,r \in [0,+\infty)$. It holds
	\begin{align*}
	\|\dot{x}(t)&-\dot{x}(r)\|+\|\dot{z}(t)-\dot{z}(r)\|\\
	&= \|u(t,x(t),z(t),y(t))-u(r,x(r),z(r),y(r))\|
	+\|v(t,x(t),z(t),y(t))-v(r,x(r),z(r),y(r))\|\\
	&\leq\|u(t,x(t),z(t),y(t))-u(t,x(r),z(r),y(r))\|+\|u(t,x(r),z(r),y(r))-u(r,x(r),z(r),y(r))\|\\
	&\quad+\|v(t,x(t),z(t),y(t))-v(t,x(r),z(r),y(r))\|+\|v(t,x(r),z(r),y(r))-v(r,x(r),z(r),y(r))\|.
	\end{align*}
	So, it follows  from \eqref{PAMADS_eq:norm_u1}, \eqref{PAMADS_eq:norm_v1}, \eqref{PAMADS_eq:norm_u2_global} and \eqref{PAMADS_eq:norm_v2_global} that there exists $L>0$ such that
	\begin{align*}
	\|\dot{x}(t)-\dot{x}(r)\|+\|\dot{z}(t)-\dot{z}(r)\|\leq &L(\|x(t)-x(r)\|+\|z(t)-z(r)\|+\|y(t)-y(r)\|\\
	&+\|\dot{x}(r)\|\|M_1(t)-M_1(r)\|+\|\dot{z}(r)\|\|M_2(t)-M_2(r)\|\\
	&+\|\dot{y}(r)\||c(t)-c(r)|)
	\end{align*}
	for every $t,r \in [0,+\infty)$. Now we fix $r \in [0,+\infty)$ at which the second derivative of the trajectories exists and take in the above inequality $t=r+h$ for some $h>0$. Then
	\begin{align*}
	\|\dot{x}(r&+h)-\dot{x}(r)\|+\|\dot{z}(r+h)-\dot{z}(r)\|\\
	&\leq L(\|x(r+h)-x(r)\|+\|z(r+h)-z(r)\|+\|y(r+h)-y(r)\|)\\
	&\quad+L(\|\dot{x}(r)\|\|M_1(r+h)-M_1(r)\|+\|\dot{z}(r)\|\|M_2(r+h)-M_2(r)\|\\
	&\quad +\|\dot{y}(r)\||c(r+h)-c(r)|).
	\end{align*}
	After dividing in the above inequality by $h$ and letting $h\to0$ we obtain
	\[
	\|\ddot{x}(r)\|+\|\ddot{z}(r)\|\leq L(\|\dot{x}(r)\|+\|\dot{z}(r)\|+\|\dot{y}(r)\|+\|\dot{x}(r)\|\|\dot{M_1}(r)\|+\|\dot{z}(r)\|\|\dot{M_2}(r)\|+\|\dot{y}(r)\||\dot{c}(r)|)
	\]
	and the proof is complete.
\end{proof}

In the following we recall two results which we need for the asymptotic analysis (see \cite[Lemma 5.1]{abat14} and \cite[Lemma 5.2]{abat14}).

\begin{lemma}\label{PAMADS_lemma:Lp_lim}
	Assume that $u:[0,+\infty)\to \R$ is locally absolutely continuous and bounded from below and that there exists $v \in L^1([0,+\infty),\R)$ with the property that for almost every $t \in [0,+\infty)$
	\[\frac{d}{dt}u(t)\leq v(t).\]
	Then there exists $\lim\limits_{t \to +\infty}u(t) \in \R$.
\end{lemma}

\begin{lemma}\label{PAMADS_lemma:Lp_lim0}
	Assume that $1 \leq p < \infty, \  1\leq r \leq \infty,  \ u:[0,+\infty)\to[0,+\infty)$ is locally absolutely continuous, $u \in L^p([0,+\infty),\R), \ v:[0,+\infty)\to \R, \ v \in L^r([0,+\infty),\R)$ and for almost every $t \in [0,+\infty)$
	\[\frac{d}{dt}u(t)\leq v(t).\]
	Then $\lim\limits_{t \to +\infty}u(t)=0$.
\end{lemma}

In the following we have the result which states the asymptotic convergence of the trajectories generated by the dynamical system \eqref{eq:proxAMA-DS} to a saddle point of the Lagrangian of the problem \eqref{opt:Prox-AMA:primal}. The derivation of the result via Lyapunov analysis is involved.

\begin{theorem}\label{PAMADS:th:convergence}
	In the setting of the optimization problem \eqref{opt:Prox-AMA:primal}, assume that the set of saddle points of the Lagrangian $L$ is nonempty, the maps
	\[
	[0,+\infty)\to S_+(\HH), t \mapsto M_1(t), \quad\text{and} ~[0,+\infty)\to S_+(\GG), t \mapsto M_2(t)
	\]
	are locally absolutely continuous and monotonically decreasing in the sense of the Loewner partial ordering defined in \eqref{Loewner},
	\[
	M_1(t)-\frac{L_{h_1}}{4}I\in S_+(\HH),
	\]
	\[
	M_2(t)-\frac{L_{h_2}}{4}I\in S_+(\GG) \quad \forall t \in [0,+\infty),
	\]
	and
	\[
	\sup_{t\geq 0}\|\dot{M}_1(t)\|<+\infty ~\text{and}~ \sup_{t\geq 0}\|\dot{M}_2(t)\|<+\infty.
	\]
	Furthermore we assume that for $0<\epsilon<\frac{\sigma}{2\|A\|^2}$ the function $c:[0,+\infty)\to \left[\epsilon, \frac{\sigma}{\|A\|^2}-\epsilon\right]$ is  monotonically decreasing and Lipschitz continuous. If $c(t)$ is a constant function, namely $c(t)=c$ for all $t \in [0,+\infty)$, then it is enough to assume that $\epsilon \leq c \leq\frac{2\sigma}{\|A\|^2}-\epsilon$.
	For an arbitrary starting point $(x^0,z^0,y^0)\in \HH \times \GG \times \mathcal{K}$, let $(x,z,y): [0,+\infty)\to\HH \times \GG \times \mathcal{K}$ be the unique strong global solution of the dynamical system \eqref{eq:proxAMA-DS}. If one of the following conditions holds:
	\begin{enumerate}
		\item there exists $\alpha>0$ such that $M_2(t)-\frac{L_{h_2}}{4}I\in P_\alpha(\GG)$ for every $t \in [0,+\infty)$
		\item there exists $\beta>0$ such that $B^*B \in P_\beta(\GG)$;
	\end{enumerate}
	then the trajectory $(x(t),z(t),y(t))$ converges weakly to a saddle point of $L$ as $t \to + \infty$.
\end{theorem}

\begin{proof}
We need an appropriate energy functional in order to conclude. This will be accomplished in \eqref{PAMADS:eq:sum3a} below.
	Let $(x^*,z^*,y^*) \in \HH \times \GG \times \mathcal{K}$ be a saddle point of the Lagrangian $L$. Then it fulfills the system of the optimality conditions
	\[
	\begin{cases}
	A^*y^*-\nabla h_1(x^*) \in \partial f(x^*)\\
	B^*y^*-\nabla h_2(z^*) \in \partial g(z^*)\\
	Ax^*+Bz^*=b.
	\end{cases}
	\]
	From \eqref{PAMADS:OC_DS_1} we have for almost every $t \in [0,+\infty)$
	\begin{equation*}
	-M_1(t)\dot{x}(t)+A^*y(t)-\nabla h_1 (x(t)) \in \partial f(\dot{x}(t)+x(t)),
	\end{equation*}
	and by taking into account the strong monotonicity of $\partial f$ we have
	\begin{equation}
	\langle -M_1(t)\dot{x}(t)+A^*(y(t)-y^*)-(\nabla h_1 (x(t))-\nabla h_1(x^*)),\dot{x}(t)+x(t)-x^*\rangle \geq \sigma \|\dot{x}(t)+x(t)-x^*\|^2.\label{PAMADS:monoton1}
	\end{equation}
	In an analog way, according to \eqref{PAMADS:OC_DS_2} we have for almost every $t \in [0,+\infty)$
	\begin{align*}
	&-c(t)B^*B(\dot{z}(t)+z(t))-M_2(t)(\dot{z}(t))+B^*y(t)-c(t)B^*A(\dot{x}(t)+x(t))+c(t)B^*b-\nabla h_2(z(t))\\
	& \in \partial g(\dot{z}(t)+z(t)),
	\end{align*}
	and by taking into account the monotonicity of $\partial g$ we have
	\begin{align}
	&\langle-c(t)B^*B(\dot{z}(t)+z(t))-M_2(t)(\dot{z}(t))+B^*(y(t)-y^*)-c(t)B^*A(\dot{x}(t)+x(t))+c(t)B^*b\nonumber\\
	&-(\nabla h_2(z(t))-\nabla h_2(z^*)),\dot{z}(t)+z(t)-z^*\rangle \geq 0. \label{PAMADS:monoton2}
	\end{align}
	We use the last equation of \eqref{eq:proxAMA-DS} and the optimality condition $Ax^*+Bz^*=b$ to obtain for almost every $t \in [0,+\infty)$
	\begin{align}
	\langle A^*&(y(t)-y^*), \dot{x}(t)+x(t)-x^*\rangle + \langle B^*(y(t)-y^*),\dot{z}(t)+z(t)-z^*\rangle\nonumber\\
	&=-\langle y(t)-y^*, -A(\dot{x}(t)+x(t))+Ax^*-B(\dot{z}(t)+z(t))+Bz^*\rangle\nonumber\\
	&=-\frac{1}{c(t)}\langle y(t)-y^*,\dot{y}(t)\rangle = -\frac{1}{2c(t)}\frac{d}{dt}\|y(t)-y^*\|^2.\label{PAMADS:eq:deriv_y}
	\end{align}
	Assume that $L_{h_1}>0$ and $L_{h_2}>0$. By using the Baillon-Haddad Theorem we know that the gradients of $h_1$ and $h_2$ are $L_1^{-1}$- and $L_2^{-1}$-cocoercive, respectively, we have for almost every $t\in[0,+\infty)$
	\begin{align}
	\langle - (\nabla h_1&(x(t))-\nabla h_1(x^*)),\dot{x}(t)+x(t)-x^*\rangle\nonumber\\
	&=-\langle\nabla h_1(x(t))-\nabla h_1(x^*),x(t)-x^*\rangle
	-\langle\nabla h_1(x(t))-\nabla h_1(x^*),\dot{x}(t)\rangle\nonumber\\
	&\leq -\frac{1}{L_{h_1}}\|\nabla h_1(x(t))-\nabla h_1(x^*)\|^2-\langle\nabla h_1(x(t))-\nabla h_1(x^*),\dot{x}(t)\rangle\nonumber\\
	&=-\frac{1}{L_{h_1}}\left(\left\|\nabla h_1(x(t))-\nabla h_1(x^*)+\frac{L_{h_1}}{2}\dot{x}(t)\right\|^2-\frac{L_{h_1}^2}{4}\left\|\dot{x}(t)\right\|^2\right)\label{PAMADS:eq:Lh1}
	\end{align}
	and respectively
	\begin{align}
	\langle - (\nabla h_2&(z(t))-\nabla h_2(z^*)),\dot{z}(t)+z(t)-z^*\rangle\nonumber\\
	&=-\frac{1}{L_{h_2}}\left(\left\|\nabla h_2(z(t))-\nabla h_2(z^*)+\frac{L_{h_2}}{2}\dot{z}(t)\right\|^2-\frac{L_{h_2}^2}{4}\left\|\dot{z}(t)\right\|^2\right).\label{PAMADS:eq:Lh2}
	\end{align}
	By summing up \eqref{PAMADS:monoton1} and \eqref{PAMADS:monoton2} and by taking into account \eqref{PAMADS:eq:deriv_y}, \eqref{PAMADS:eq:Lh1} and \eqref{PAMADS:eq:Lh2},
	we obtain for almost every $t\in[0,+\infty)$
	\begin{align}
	0\leq& \langle -M_1(t)\dot{x}(t),\dot{x}(t)+x(t)-x^*\rangle+\langle-c(t)B^*B(\dot{z}(t)+z(t))-M_2(t)(\dot{z}(t))\nonumber\\
	&-c(t)B^*A(\dot{x}(t)+x(t))+c(t)B^*b,\dot{z}(t)+z(t)-z^*\rangle-\frac{1}{2c(t)}\frac{d}{dt}\|y(t)-y^*\|^2\nonumber\\
	&-\frac{1}{L_{h_1}}\left(\left\|\nabla h_1(x(t))-\nabla h_1(x^*)+\frac{L_{h_1}}{2}\dot{x}(t)\right\|^2-\frac{L_{h_1}^2}{4}\left\|\dot{x}(t)\right\|^2\right)\nonumber\\
	&-\frac{1}{L_{h_2}}\left(\left\|\nabla h_2(z(t))-\nabla h_2(z^*)+\frac{L_{h_2}}{2}\dot{z}(t)\right\|^2-\frac{L_{h_2}^2}{4}\left\|\dot{z}(t)\right\|^2\right)-\sigma \|\dot{x}(t)+x(t)-x^*\|^2.\label{PAMADS:eq:sum1}
	\end{align}

	We have for almost every $t\in[0,+\infty)$ (use also the last equality for $\dot y$ in \eqref{eq:proxAMA-DS}):
	\begin{align}
	\langle-c(t)B&^*B(\dot{z}(t)+z(t))-c(t)B^*A(\dot{x}(t)+x(t))+c(t)B^*b,\dot{z}(t)+z(t)-z^*\rangle-\sigma \|\dot{x}(t)+x(t)-x^*\|^2\nonumber\\
	&=-\frac{1}{c(t)}\langle\dot{y}(t),-c(t)B(\dot{z}(t)+z(t)-z^*)\rangle-\sigma \|\dot{x}(t)+x(t)-x^*\|^2\nonumber\\
	&=-\frac{1}{c(t)}\left[\frac{1}{2}\|\dot{y}(t)\|^2+\frac{1}{2}\|c(t)B(\dot{z}(t)+z(t)-z^*)\|^2-\frac{1}{2}\|\dot{y}(t)+c(t)B(\dot{z}(t)+z(t)-z^*)\|^2\right]\nonumber\\
	&\quad-\sigma \|\dot{x}(t)+x(t)-x^*\|^2\nonumber\\
	&=-\frac{1}{c(t)}\left[\frac{1}{2}\|\dot{y}(t)\|^2+\frac{1}{2}c^2(t)\left[\|B(z(t)-z^*)\|^2+\|B\dot{z}(t)\|^2+2\langle\dot{z}(t),B^*B(z(t)-z^*)\rangle\right]\right.\nonumber\\
	&\quad\left.-\frac{1}{2}\|c(t)(b-A(x(t)+\dot{x}(t))-Bz^*)\|^2\right]
	-\sigma \|\dot{x}(t)+x(t)-x^*\|^2\nonumber\\
	&=-\frac{1}{c(t)}\left[\frac{1}{2}\|\dot{y}(t)\|^2+\frac{1}{2}c^2(t)\left[\|Bz(t)-Bz^*\|^2+\|B\dot{z}(t)\|^2+\frac{d}{dt}\|Bz(t)-Bz^*\|^2\right]\right.\nonumber\\
	&\quad\left.-\frac{1}{2}\|c(t)(-A(x(t)+\dot{x}(t))+Ax^*)\|^2\right]
	\quad-\sigma \|\dot{x}(t)+x(t)-x^*\|^2\nonumber\\
	&\leq -\frac{1}{c(t)}\left[\frac{1}{2}\|\dot{y}(t)\|^2+\frac{1}{2}c^2(t)\left[\|Bz(t)-Bz^*\|^2+\|B\dot{z}(t)\|^2+\frac{d}{dt}\|Bz(t)-Bz^*\|^2\right]\right]\nonumber\\
	&\quad+\left(\frac{1}{2}c(t)\|A\|^2-\sigma\right) \|\dot{x}(t)+x(t)-x^*\|^2\nonumber\\
	&=-\frac{1}{2c(t)}\|\dot{y}(t)\|^2-\frac{c(t)}{2}\|Bz(t)-Bz^*\|^2-\frac{c(t)}{2}\|B\dot{z}(t)\|^2-\frac{c(t)}{2}\frac{d}{dt}\|Bz(t)-Bz^*\|^2\nonumber\\
	&\quad+\left(\frac{1}{2}c(t)\|A\|^2-\sigma\right) \left(\|x(t)-x^*\|^2+\|\dot{x}(t)\|^2+\frac{d}{dt}\|x(t)-x^*\|^2\right).\label{PAMADS:eq:sum2a}
	\end{align}
	By using \eqref{diff} we observe that for almost every $t\in [0,+\infty)$ it holds
	\begin{align*}
	\langle -M_1(t)\dot{x}(t),\dot{x}(t)+x(t)-x^*\rangle&=-\|\dot{x}(t)\|^2_{M_1(t)}-\langle M_1(t)\dot{x}(t),x(t)-x^*\rangle\\
	&=-\|\dot{x}(t)\|^2_{M_1(t)}+\frac{1}{2}\langle \dot{M_1}(t)(x(t)-x^*),x(t)-x^*\rangle\\
	&\quad-\frac{1}{2}\frac{d}{dt}\|x(t)-x^*\|^2_{M_1(t)}
	\end{align*}
	and
	\begin{align*}
	\langle -M_2(t)\dot{z}(t),\dot{z}(t)+z(t)-z^*\rangle&=-\|\dot{z}(t)\|^2_{M_2(t)}-\langle M_2(t)\dot{z}(t),z(t)-z^*\rangle\\
	&=-\|\dot{z}(t)\|^2_{M_2(t)}+\frac{1}{2}\langle \dot{M_2}(t)(z(t)-z^*),z(t)-z^*\rangle\\
	&\quad-\frac{1}{2}\frac{d}{dt}\|z(t)-z^*\|^2_{M_2(t)}.
	\end{align*}
	By plugging the last two identities and \eqref{PAMADS:eq:sum2a} into \eqref{PAMADS:eq:sum1}, we obtain for almost every $t\in [0,+\infty)$
	\begin{align}
	0 &\leq -\frac{1}{2c(t)}\|\dot{y}(t)\|^2-\frac{c(t)}{2}\|B\dot{z}(t)\|^2-\left(\sigma-\frac{1}{2}c(t)\|A\|^2\right) (\|x(t)-x^*\|^2+\|\dot{x}(t)\|^2)-\frac{c(t)}{2}\|Bz(t)-Bz^*\|^2\nonumber\\
	&\quad-\frac{1}{2}\left((2\sigma-c(t)\|A\|^2) \frac{d}{dt}\|x(t)-x^*\|^2+\frac{d}{dt}\|x(t)-x^*\|_{M_1(t)}^2+c(t)\frac{d}{dt}\|Bz(t)-Bz^*\|^2\right.\nonumber\\
	&\quad\left.+\frac{d}{dt}\|z(t)-z^*\|_{M_2(t)}^2+\frac{1}{c(t)}\frac{d}{dt}\|y(t)-y^*\|^2\right)-\|\dot{x}(t)\|^2_{M_1(t)}+\frac{1}{2}\langle \dot{M_1}(t)(x(t)-x^*),x(t)-x^*\rangle\nonumber\\
	&\quad-\|\dot{z}(t)\|^2_{M_2(t)}+\frac{1}{2}\langle \dot{M_2}(t)(z(t)-z^*),z(t)-z^*\rangle-\frac{1}{L_{h_1}}\left\|\nabla h_1(x(t))-\nabla h_1(x^*)+\frac{L_{h_1}}{2}\dot{x}(t)\right\|^2\nonumber\\
	&\quad+\frac{L_{h_1}}{4}\|\dot{x}(t)\|^2-\frac{1}{L_{h_2}}\left\|\nabla h_2(z(t))-\nabla h_2(z^*)+\frac{L_{h_2}}{2}\dot{z}(t)\right\|^2+\frac{L_{h_2}}{4}\|\dot{z}(t)\|^2\nonumber.
	\end{align}
	Taking into account that
	\begin{align*}
	&-\frac{1}{2}\left((2\sigma-c(t)\|A\|^2) \frac{d}{dt}\|x(t)-x^*\|^2+\frac{d}{dt}\|x(t)-x^*\|_{M_1(t)}^2+c(t)\frac{d}{dt}\|Bz(t)-Bz^*\|^2\right.\nonumber\\
	&\quad\left.+\frac{d}{dt}\|z(t)-z^*\|_{M_2(t)}^2+\frac{1}{c(t)}\frac{d}{dt}\|y(t)-y^*\|^2\right)\\
	=&-\frac{1}{2c(t)}\left(\frac{d}{dt}(2\sigma c(t)-c^2(t)\|A\|^2)\|x(t)-x^*\|^2-(2\dot{c}(t)\sigma-2c(t)\dot{c}(t)\|A\|^2)\|x(t)-x^*\|^2\right.\\
	&\left.+\frac{d}{dt}(c(t)\|x(t)-x^*\|_{M_1(t)}^2)-\dot{c}(t)\|x(t)-x^*\|_{M_1(t)}^2+\frac{d}{dt}(c^2(t)\|Bz(t)-Bz^*\|^2)\right.\\
	&\left.-2c(t)\dot{c}(t)\|Bz(t)-Bz^*\|^2+\frac{d}{dt}(c(t)\|z(t)-z^*\|_{M_2(t)}^2)-\dot{c}(t)\|z(t)-z^*\|_{M_2(t)}^2+\frac{d}{dt}\|y(t)-y^*\|^2\right)\\
	=&-\frac{1}{2c(t)}\frac{d}{dt}\left((2\sigma c(t)-c^2(t)\|A\|^2)\|x(t)-x^*\|^2+\|x(t)-x^*\|_{c(t)M_1(t)}^2\right.\\
	&\left.+\|z(t)-z^*\|_{c(t)M_2(t)+c^2(t)B^*B}^2+\|y(t)-y^*\|^2\right)\\
	&+\frac{\dot{c}(t)}{2c(t)}\left(2(\sigma-c(t)\|A\|^2)\|x(t)-x^*\|^2+\|x(t)-x^*\|_{M_1(t)}^2+2c(t)\|Bz(t)-Bz^*\|^2+\|z(t)-z^*\|_{M_2(t)}^2\right),
	\end{align*}
	we obtain that
	\begin{align}
	0 &\leq
	-\frac{1}{2c(t)}\frac{d}{dt}\left((2\sigma c(t)-c^2(t)\|A\|^2)\|x(t)-x^*\|^2+\|x(t)-x^*\|_{c(t)M_1(t)}^2\right.\nonumber\\
	&\left.\quad+\|z(t)-z^*\|_{c(t)M_2(t)+c^2(t)B^*B}^2+\|y(t)-y^*\|^2\right)\nonumber\\
	&\quad+\frac{\dot{c}(t)}{2c(t)}\left(2(\sigma-c(t)\|A\|^2)\|x(t)-x^*\|^2+\|x(t)-x^*\|_{M_1(t)}^2+2c(t)\|Bz(t)-Bz^*\|^2+\|z(t)-z^*\|_{M_2(t)}^2\right)\nonumber\\
	&\quad-\frac{1}{2c(t)}\|\dot{y}
	(t)\|^2-\frac{c(t)}{2}\|B\dot{z}(t)\|^2-\left(\sigma-\frac{1}{2}c(t)\|A\|^2\right) (\|x(t)-x^*\|^2+\|\dot{x}(t)\|^2)-\frac{c(t)}{2}\|Bz(t)-Bz^*\|^2\nonumber\\
	&\quad-\|\dot{x}(t)\|^2_{M_1(t)}+\frac{1}{2}\langle \dot{M_1}(t)(x(t)-x^*),x(t)-x^*\rangle\nonumber\\
	&\quad-\|\dot{z}(t)\|^2_{M_2(t)}+\frac{1}{2}\langle \dot{M_2}(t)(z(t)-z^*),z(t)-z^*\rangle-\frac{1}{L_{h_1}}\left\|\nabla h_1(x(t))-\nabla h_1(x^*)+\frac{L_{h_1}}{2}\dot{x}(t)\right\|^2\nonumber\\
	&\quad+\frac{L_{h_1}}{4}\|\dot{x}(t)\|^2-\frac{1}{L_{h_2}}\left\|\nabla h_2(z(t))-\nabla h_2(z^*)+\frac{L_{h_2}}{2}\dot{z}(t)\right\|^2+\frac{L_{h_2}}{4}\|\dot{z}(t)\|^2\nonumber.
	\end{align}
	Since $\dot{c}(t)\leq 0$, $0 < c(t) \leq \frac{\sigma}{\|A\|^2}$ if $c(t)$ is not constant (if $c(t)$ is constant we have $\dot{c}(t)=0$) and $\langle \dot{M_1}(t)(x(t)-x^*),x(t)-x^*\rangle\leq 0$ and $\langle \dot{M_2}(t)(z(t)-z^*),z(t)-z^*\rangle\leq 0$ (which follows easily from Definition \ref{def-mdot} and the decreasing property of $M_1$ and $M_2$) we have for almost every $t\in [0,+\infty)$
	\begin{align}
	0 &\geq \frac{1}{2}\frac{d}{dt}\left((2\sigma c(t)-c^2(t)\|A\|^2)\|x(t)-x^*\|^2+\|x(t)-x^*\|_{c(t)M_1(t)}^2\right.\nonumber\\
	&\left.\quad+\|z(t)-z^*\|_{c(t)M_2(t)+c^2(t)B^*B}^2+\|y(t)-y^*\|^2\right)\nonumber\\
	&\quad+c(t)\|\dot{x}(t)\|^2_{M_1(t)-\frac{L_{h_1}}{4}I}+c(t)\left(\sigma-\frac{1}{2}c(t)\|A\|^2\right)\|\dot{x}(t)\|^2+c(t)\|\dot{z}(t)\|^2_{M_2(t)+\frac{c(t)}{2}B^*B-\frac{L_{h_2}}{4}I}+\frac{1}{2}\|\dot{y}(t)\|^2\nonumber\\
	&\quad+c(t)\left(\sigma-\frac{1}{2}c(t)\|A\|^2\right) \|x(t)-x^*\|^2+\frac{c^2(t)}{2}\|Bz(t)-Bz^*\|^2\nonumber\\
	&\quad+\frac{c(t)}{L_{h_1}}\left\|\nabla h_1(x(t))-\nabla h_1(x^*)+\frac{L_{h_1}}{2}\dot{x}(t)\right\|^2+\frac{c(t)}{L_{h_2}}\left\|\nabla h_2(z(t))-\nabla h_2(z^*)+\frac{L_{h_2}}{2}\dot{z}(t)\right\|^2\nonumber.
	\end{align}
	For $\underline{c}:=\epsilon$ and $\overline{c}:=\frac{2\sigma}{\|A\|^2}-\epsilon$ we have
	\begin{align}
	0 &\geq \frac{1}{2}\frac{d}{dt}\left((2\sigma c(t)-c^2(t)\|A\|^2)\|x(t)-x^*\|^2+\|x(t)-x^*\|_{c(t)M_1(t)}^2\right.\nonumber\\
	&\left.\quad+\|z(t)-z^*\|_{c(t)M_2(t)+c^2(t)B^*B}^2+\|y(t)-y^*\|^2\right)\nonumber\\
	&\quad+\underline{c}\|\dot{x}(t)\|^2_{M_1(t)-\frac{L_{h_1}}{4}I}+\underline{c}\left(\sigma-\frac{1}{2}\overline{c}\|A\|^2\right)\|\dot{x}(t)\|^2+\underline{c}\|\dot{z}(t)\|^2_{M_2(t)+\frac{c(t)}{2}B^*B-\frac{L_{h_2}}{4}I}+\frac{1}{2}\|\dot{y}(t)\|^2\nonumber\\
	&\quad+\underline{c}\left(\sigma-\frac{1}{2}\overline{c}\|A\|^2\right) \|x(t)-x^*\|^2+\frac{\underline{c}^2}{2}\|Bz(t)-Bz^*\|^2\nonumber\\
	&\quad+\frac{\underline{c}}{L_{h_1}}\left\|\nabla h_1(x(t))-\nabla h_1(x^*)+\frac{L_{h_1}}{2}\dot{x}(t)\right\|^2+\frac{\underline{c}}{L_{h_2}}\left\|\nabla h_2(z(t))-\nabla h_2(z^*)+\frac{L_{h_2}}{2}\dot{z}(t)\right\|^2\label{PAMADS:eq:sum3a}.
	\end{align}
	From Lemma ~\ref{PAMADS_lemma:Lp_lim}~ we have
	\begin{align}
	\exists \lim\limits_{t \to + \infty}\left((2\sigma c(t)-c^2(t)\|A\|^2)\|x(t)-x^*\|^2+\|x(t)-x^*\|_{c(t)M_1(t)}^2\right.\nonumber\\
	\left.+\|z(t)-z^*\|_{c(t)M_2(t)+c^2(t)B^*B}^2+\|y(t)-y^*\|^2\right) \in \R.\label{PAMADS:lim_xzy}
	\end{align}
	Let $T>0$. By integrating \eqref{PAMADS:eq:sum3a} on the interval $[0,T]$ we obtain
	\begin{align*}
	\frac{1}{2}&\left((2\sigma c(T)-c^2(T)\|A\|^2)\|x(T)-x^*\|^2+\|x(T)-x^*\|_{c(T)M_1(T)}^2\right.\\
	&\left.+\|z(T)-z^*\|_{c(T)M_2(T)+c^2(T)B^*B}^2+\|y(T)-y^*\|^2\right)\\
	&+\underline{c}\int_{0}^{T}\|\dot{x}(t)\|^2_{M_1(t)-\frac{L_{h_1}}{4}I}dt+\underline{c}\left(\sigma-\frac{1}{2}\overline{c}\|A\|^2\right)\int_{0}^{T}\|\dot{x}(t)\|^2dt+\underline{c}\int_{0}^{T}\|\dot{z}(t)\|^2_{M_2(t)+\frac{c(t)}{2}B^*B-\frac{L_{h_2}}{4}I}dt\\
	&\quad+\frac{1}{2}\int_{0}^{T}\|\dot{y}(t)\|^2dt+ \underline{c}\left(\sigma-\frac{1}{2}\overline{c}\|A\|^2\right)\int_{0}^{T}\|x(t)-x^*\|^2dt+\frac{\underline{c}^2}{2}\int_{0}^{T}\|Bz(t)-Bz^*\|^2dt\\
	&+\frac{\underline{c}}{L_{h_1}}\int_{0}^{T}\left\|\nabla h_1(x(t))-\nabla h_1(x^*)+\frac{L_{h_1}}{2}\dot{x}(t)\right\|^2dt+\frac{\underline{c}}{L_{h_2}}\int_{0}^{T}\left\|\nabla h_2(z(t))+\nabla h_2(z^*)+\frac{L_{h_2}}{2}\dot{z}(t)\right\|^2dt\\
	&\leq\frac{1}{2}\left((2\sigma c(0)-c^2(0)\|A\|^2)\|x_0-x^*\|^2+\|x_0-x^*\|_{c(0)M_1(0)}^2+\|z_0-z^*\|_{c(0)M_2(0)+c^2(0)B^*B}^2+\|y_0-y^*\|^2\right).
	\end{align*}
	Letting $T$ converge to $+\infty$ we have
	\begin{align}
	\|\dot{x}(\cdot)\|^2_{M_1(\cdot)-\frac{L_{h_1}}{4}I}\in L^1([0,+\infty),\R), \quad 	\|\dot{x}(\cdot)\|^2\in L^1([0,+\infty),\R),\label{PAMADS:eq_conv_x_dot}\\
	\|\dot{z}(\cdot)\|^2_{M_2(\cdot)+\frac{c(\cdot)}{2}B^*B-\frac{L_{h_2}}{4}I}\in L^1([0,+\infty),\R),\label{PAMADS:eq_conv_z_dot}\\
	\dot{y}(\cdot) \in  L^2([0,+\infty),\mathcal{K}),\label{PAMADS:eq_conv_y_dot}\\
	 x(\cdot)-x^* \in L^2([0,+\infty),\HH), \quad Bz(\cdot)-Bz^*\in L^2([0,+\infty),\HH).\label{PAMADS:eq_conv_x^*_Bz^*}
	\end{align}
	In the case when $L_{h_1}=0$ and $L_{h_2}>0$, we have that $\nabla h_1$ is constant and instead of \eqref{PAMADS:eq:sum3a} we obtain for almost every $t \in [0,+\infty)$
	\begin{align}
	0 &\geq \frac{1}{2}\frac{d}{dt}\left((2\sigma c(t)-c^2(t)\|A\|^2)\|x(t)-x^*\|^2+\|x(t)-x^*\|_{c(t)M_1(t)}^2\right.\nonumber\\
	&\left.\quad+\|z(t)-z^*\|_{c(t)M_2(t)+c^2(t)B^*B}^2+\|y(t)-y^*\|^2\right)\nonumber\\
	&\quad+\underline{c}\|\dot{x}(t)\|^2_{M_1(t)}+\underline{c}\left(\sigma-\frac{1}{2}\overline{c}\|A\|^2\right)\|\dot{x}(t)\|^2+\underline{c}\|\dot{z}(t)\|^2_{M_2(t)+\frac{c(t)}{2}B^*B-\frac{L_{h_2}}{4}I}+\frac{1}{2}\|\dot{y}(t)\|^2\nonumber\\
	&\quad+\underline{c}\left(\sigma-\frac{1}{2}\overline{c}\|A\|^2\right) \|x(t)-x^*\|^2+\frac{\underline{c}^2}{2}\|Bz(t)-Bz^*\|^2\nonumber\\
	&\quad+\frac{\underline{c}}{L_{h_2}}\left\|\nabla h_2(z(t))-\nabla h_2(z^*)+\frac{L_{h_2}}{2}\dot{z}(t)\right\|^2\label{PAMADS:eq:sum3b}.
	\end{align}
	Similarly, in the case when $L_{h_1}>0$ and $L_{h_2}=0$ we obtain for almost every $t \in [0,+\infty)$
	\begin{align}
	0 &\geq \frac{1}{2}\frac{d}{dt}\left((2\sigma c(t)-c^2(t)\|A\|^2)\|x(t)-x^*\|^2+\|x(t)-x^*\|_{c(t)M_1(t)-\frac{L_{h_1}}{4}I}^2\right.\nonumber\\
	&\left.\quad+\|z(t)-z^*\|_{c(t)M_2(t)+c^2(t)B^*B}^2+\|y(t)-y^*\|^2\right)\nonumber\\
	&\quad+\underline{c}\|\dot{x}(t)\|^2_{M_1(t)}+\underline{c}\left(\sigma-\frac{1}{2}\overline{c}\|A\|^2\right)\|\dot{x}(t)\|^2+\underline{c}\|\dot{z}(t)\|^2_{M_2(t)+\frac{c(t)}{2}B^*B}+\frac{1}{2}\|\dot{y}(t)\|^2\nonumber\\
	&\quad+\underline{c}\left(\sigma-\frac{1}{2}\overline{c}\|A\|^2\right) \|x(t)-x^*\|^2+\frac{\underline{c}^2}{2}\|Bz(t)-Bz^*\|^2\nonumber\\
	&\quad+\frac{\underline{c}}{L_{h_1}}\left\|\nabla h_1(x(t))-\nabla h_1(x^*)+\frac{L_{h_1}}{2}\dot{x}(t)\right\|^2\label{PAMADS:eq:sum3c}
	\end{align}
	and in the case when $L_{h_1}=0$ and $L_{h_2}=0$ we obtain for almost every $t \in [0,+\infty)$
	\begin{align}
	0 &\geq \frac{1}{2}\frac{d}{dt}\left((2\sigma c(t)-c^2(t)\|A\|^2)\|x(t)-x^*\|^2+\|x(t)-x^*\|_{c(t)M_1(t)}^2\right.\nonumber\\
	&\left.\quad+\|z(t)-z^*\|_{c(t)M_2(t)+c^2(t)B^*B}^2+\|y(t)-y^*\|^2\right)\nonumber\\
	&\quad+\underline{c}\|\dot{x}(t)\|^2_{M_1(t)}+\underline{c}\left(\sigma-\frac{1}{2}\overline{c}\|A\|^2\right)\|\dot{x}(t)\|^2+\underline{c}\|\dot{z}(t)\|^2_{M_2(t)+\frac{c(t)}{2}B^*B}+\frac{1}{2}\|\dot{y}(t)\|^2\nonumber\\
	&\quad+\underline{c}\left(\sigma-\frac{1}{2}\overline{c}\|A\|^2\right) \|x(t)-x^*\|^2+\frac{\underline{c}^2}{2}\|Bz(t)-Bz^*\|^2\label{PAMADS:eq:sum3d}.
	\end{align}
	By arguing as above, we obtain also in these three cases that \eqref{PAMADS:lim_xzy} and \eqref{PAMADS:eq_conv_x_dot}-\eqref{PAMADS:eq_conv_x^*_Bz^*} hold.
	
	We can easily see that, if assumptions 1. or 2. from the theorem hold true, then we have 
	$\dot{z}(\cdot) \in  L^2([0,+\infty),\GG)$.
	Further, taking into acount the hypotheses concerning $\dot M_1,\dot M_2$ and $c$, we can easily 
	derive from Lemma ~\ref{PAMADS_lemma:xz_dot}~ that 
	\[\ddot{x}(\cdot) \in  L^2([0,+\infty),\HH) ~\text{and}~\ddot{z}(\cdot) \in  L^2([0,+\infty),\GG).\]
	It follows, for almost every $t \in [0,+\infty)$
	\[\frac{d}{dt}\|\dot{x}(t)\|^2=2\langle\ddot{x}(t),\dot{x}(t)\rangle \leq (\|\ddot{x}(t)\|^2+\|\dot{x}(t)\|^2) \]
	and the right-hand side is a function in $L^1([0,+\infty),\R)$. By Lemma ~\ref{PAMADS_lemma:Lp_lim0} we have
	\[\lim\limits_{t \to + \infty}\dot{x}(t)=0.\]
	Similarly, we obtain that
	\[\lim\limits_{t \to + \infty}\dot{z}(t)=0 \quad\lim\limits_{t \rightarrow +\infty}(x(t)-x^*)=0 ~ \text{and} ~\lim\limits_{t \rightarrow +\infty}(Bz(t)-Bz^*)=0.\]
	Because $\lim\limits_{t \to + \infty}\frac{1}{c(t)}\dot{y}(t)=\lim\limits_{t \to + \infty}(b-A(x(t)+\dot{x}(t))-B(z(t)+\dot{z}(t))) = b-Ax^*-Bz^*$  and the optimality condition $Ax^*+Bz^*=b$, we have
	\[\lim\limits_{t \to + \infty}\dot{y}(t)=0.\]
	
	In the following, let us prove that each weak sequential cluster point of $(x(t),z(t),y(t)),~t \in [0,+\infty)$ is a saddle point of $L$ (notice that the trajectories are bounded). Let $(x^*,\overline{z},\overline{y})$ be such a weak sequentially cluster point. This means that there exists a sequence $(s_n)_{n\geq0}$ with $s_n \rightarrow +\infty$ such that $(x(s_n),z(s_n),y(s_n))$ converges to $(x^*,\overline{z},\overline{y})$ as $n \rightarrow +\infty$ in the weak topology of $\HH \times \GG \times \mathcal{K}$ (notice that the trajectory $x(t)$ converges to $x^*$ strongly).
	
	From \eqref{PAMADS:OC_DS_1} we have $  \forall t \in [0, +\infty)$
	\[-M_1(s_n)\dot{x}(s_n)+A^*y(s_n)-\nabla h_1(x(s_n))\in ~\partial f(\dot{x}(s_n)+x(s_n)).\]
	
	Since $(M_1(s_n))_{n \geq 0}$ is bounded, $\nabla h_1$ is continuous, $(y(s_n))_{n \geq 0}$ converges weakly to $\overline{y}$, $\lim\limits_{t \rightarrow +\infty}\dot{x}(t)=0$ and $\lim\limits_{t \rightarrow +\infty}x(t)=x^*$, it follows from Proposition 20.33 in \cite{baco17}
	\[A^*\overline{y}-\nabla h_1(x^*)\in \partial f(x^*).\]
	From \eqref{PAMADS:OC_DS_2}, we get for every $n \geq 0$
	\[B^*\dot{y}(s_n)-M_2(s_n)\dot{z}(s_n)+B^*y(s_n)-\nabla h_2(z(s_n))+\nabla h_2(\dot{z}(s_n)+z(s_n))\in \partial (g+h_2)(\dot{z}(s_n)+z(s_n)),\]
	which is equivalent to
	\[\dot{z}(s_n)+z(s_n) \in\partial (g+h_2)^*(B^*(\dot{y}(s_n)+y(s_n))-M_2(s_n)\dot{z}(s_n)-\nabla h_2(z(s_n))+\nabla h_2(\dot{z}(s_n)+z(s_n))). \]
	By denoting for all $n \geq 0$
	\begin{align*}
	v_n &:= \dot{z}(s_n)+z(s_n),~u_n:=\dot{y}(s_n)+y(s_n)\\
	w_n &:=-M_2(s_n)\dot{z}(s_n)-\nabla h_2(z(s_n))+\nabla h_2(\dot{z}(s_n)+z(s_n)),
	\end{align*}
	we obtain
	\[v_n \in \partial (g+h_2)^*(B^*u_n+w_n) .\]
	Since $\nabla h_2$ is Lipschitz continuous, we have
	\begin{equation*}\label{PAMADS:eq_conv_h_2}
	\nabla h_2(\dot{x}(s_n)+x(s_n))-\nabla h_2(x(s_n)) \rightarrow 0 ~(n \to + \infty).
	\end{equation*}
	According to this fact and \eqref{PAMADS:eq_conv_x^*_Bz^*}, we have $v_n \rightharpoonup \overline{z}$, $u_n \rightharpoonup \overline{y}$, $Bv_n \to B\overline{z}=Bz^*$ and $w_n \to 0$ as $n \to +\infty$. Due to the monotonicity of the subdifferential, we have for all $(u,v)$ in the graph of $\partial (g+h_2)^*$ and for all $n\geq 0$
	\[\langle Bv_n-Bv, u_n \rangle + \langle v_n-v, w_n-u \rangle \geq 0.\]
	We let $n$ converge to $+\infty$ and obtain
	\[\langle \overline{z}-v, B^*\overline{y} -u\rangle \geq 0 ~~\forall (u,v) \text{ in the graph of }\partial (g+h_2)^*.\]	
	The maximal monotonicity of the convex subdifferential of $\partial (g+h_2)^*$ ensures that $\overline{z} \in \partial (g+h_2)^*(B^*\overline{y})$, which is equivalent to $B^*\overline{y}\in \partial(g+h_2)(\overline{z})$. So we have $B^*\overline{y}-\nabla h_2(\overline{z})\in \partial g(\overline{z})$.
	From \eqref{eq:proxAMA-DS} we have
	\[b-A(\dot{x}(s_n)+x(s_n))-B(\dot{z}(s_n)+z(s_n))=\frac{1}{c(s_n)}\dot{y}(s_n)\rightarrow 0~~(n \to \infty)\]
	and so it follows that $A\overline{x}+B\overline{z}=b$. In conclusion, $(x^*,\overline{z}, \overline{y})$ is a saddle point of the Lagrangian $L$.

	In the following, we show that $(x(t),z(t),y(t)),~t \in [0,+\infty)$ converges weakly. So we consider two sequential cluster points $(x^*,z_1,y_1)$ and $(x^*,z_2,y_2)$. Consequently, there exists $(k_n)_{n\geq0}$ and  $(l_n)_{n\geq0}$, such that the subsequence $(x(k_n),z(k_n),y(k_n))$ converges weakly to $(x^*,z_1,y_1)$ as $n \to +\infty$ and $(x(l_n),z(l_n),y(l_n))$ converges weakly to $(x^*,z_2,y_2)$ as $n \to +\infty$, respectively. As seen before,  $(x^*,z_1,y_1)$ and $(x^*,z_2,y_2)$ are both saddle points of the Lagrangian $L$. From \eqref{PAMADS:lim_xzy}, which is fulfilled for every saddle point of the Lagrangian $L$, we obtain
	\begin{equation}\label{PAMADS:lim_z1z2y1y2}
	\exists\lim\limits_{t \to + \infty}\left(\|z(t)-z_1\|_{c(t)M_2(t)+c^2(t)B^*B}^2-\|z(t)-z_2\|_{c(t)M_2(t)+c^2(t)B^*B}^2+\|y(t)-y_1\|^2-\|y(t)-y_2\|^2\right):= T.
	\end{equation}
	For $t \in [0,+\infty)$, we have
	\begin{align*}
	\|z(t)-&z_1\|_{c(t)M_2(t)+c^2(t)B^*B}^2-\|z(t)-z_2\|_{c(t)M_2(t)+c^2(t)B^*B}^2+\|y(t)-y_1\|^2-\|y(t)-y_2\|^2\\
	=&\|z_2-z_1\|_{c(t)M_2(t)+c^2(t)B^*B}^2+2\langle z(t)-z_2,z_2-z_1 \rangle_{c(t)M_2(t)+c^2(t)B^*B}\\
	&+\|y_2-y_1\|^2+\langle y(t)-y_2,y_2-y_1\rangle.
	\end{align*}
	Since $c(t)M_2(t)+c^2(t)B^*B$ is monotonically decreasing and positive definite, there exists a positive definite operator $M$ such that $c(t)M_2(t)+c^2(t)B^*B$ converges to $M$ in  the strong topology as $t \to + \infty$. Furthermore, let $c:=\lim\limits_{t \to + \infty}c(t)>0$. Taking the limits in \eqref{PAMADS:lim_z1z2y1y2} along the subsequences $(k_n)_{n\geq0}$ and  $(l_n)_{n\geq0}$ it yields
	\[T=-\|z_2-z_1\|_M^2-\|y_2-y_1\|^2=\|z_2-z_1\|_M^2+\|y_2-y_1\|^2,\]
	so that
	\[\|z_2-z_1\|_M^2+\|y_2-y_1\|^2=0.\]
	It follows that $z_1=z_2$ and $y_1=y_2$. In consequence, $(x(t),z(t),y(t))$ converges weakly to a saddle point of the Lagrangian $L$.

\end{proof}

In the following corollary we set for every $t\in [0,+\infty)$ $M_1(t)=0$ and $M_2(t)=\frac{1}{\tau(t)}\text{Id}-c(t)B^*B$,
where $\tau(t)>0$ and $\tau(t)c(t)\|B\|^2 \leq 1$, like in Remark \ref*{PAMADS:re:M_2(t)}. Then we get the following convergence result for the trajectory $(x(t),z(t),y(t))$ of the dynamical system \eqref{eq:proxAMA-DS_special_case} as a special case of Theorem \ref{PAMADS:th:convergence}:

\begin{corollary}\label{PAMADS:co:convergence_special_case}
In the setting of the optimization problem \eqref{opt:Prox-AMA:primal}, assume that the set of saddle points of the Lagrangian $L$ is nonempty, the map $\tau:[0,+\infty) \to (0,+\infty)$ is locally absolutely continuous, monotonically increasing and fulfills $\sup_{t\geq 0}\frac{\tau'(t)}{\tau(t)^2}<\infty$. Furthermore we assume that for an $\epsilon>0$ the map $c:[0,+\infty)\to \left[\epsilon, \frac{\sigma}{\|A\|^2}-\epsilon\right]$ is monotonically decreasing and Lipschitz continuous. If $c(t)$ is a constant function, namely $c(t)=c$ for all $t \in [0,+\infty)$, then its enough to assume that $\epsilon \leq c \leq\frac{2\sigma}{\|A\|^2}-\epsilon$. Furthermore we assume that
\begin{equation}\label{PAMADS:eq:co:convergence_special_case}
c(t)\tau(t)\|B\|^2 \leq 1-\frac{\tau(t)}{4}L_{h_2}, ~ -c'(t)\|B\|^2 \leq \frac{\tau'(t)}{\tau(t)^2}
\end{equation}
for all $ t \in [0,+\infty)$. For an arbitrary starting point $(x^0,z^0,y^0)\in \HH \times \GG \times \mathcal{K}$, let $(x,z,y): [0,+\infty)\to\HH \times \GG \times \mathcal{K}$ be the unique strong global solution of the dynamical system \eqref{eq:proxAMA-DS_special_case}. If one of the following conditions holds:
\begin{enumerate}
	\item $c(t)\tau(t)\|B\|^2 < 1-\frac{\tau(t)}{4}L_{h_2}$ for all $ t \in [0,+\infty)$
	\item there exists $\beta>0$ such that $B^*B \in P_\beta(\GG)$;
\end{enumerate}
then the trajectory $(x(t),z(t),y(t))$ converges weakly to a saddle point of $L$ as $t \to + \infty$.	
\end{corollary}

\begin{remark}\label{rem-ex}
	An appropriate choice for $\tau(t)$ to fulfill the assumptions \eqref{PAMADS:eq:co:convergence_special_case} is for example $\tau(t)=\frac{a}{c(t)}$, where $0 < a \leq \frac{1}{\|B\|^2}$. Thats how we set it in Example \ref{PAMADS:example}.
\end{remark}

If $h_1=0$ and $h_2=0$,  and $M_1(t)=0$ and $M_2(t) = 0$ for all $t \geq 0$,  then the dynamical system \eqref{eq:proxAMA-DS} becomes a continuous version of the AMA method proposed by Tseng in \cite{tseng91} which can be written as

\begin{equation*}
\begin{cases}
\dot{x}(t)+x(t) = \text{argmin}_{x \in  \HH} \left\{f(x)-\langle y(t),Ax(t)\rangle\right\}\\[2ex]
\dot{z}(t)+z(t)\in \text{argmin}_{z \in  \GG} \{g(z)-\langle y(t),Bz\rangle+\frac{c(t)}{2}\|A(x(t)+\dot{x}(t))+Bz-b\|^2\}\\[2ex]
\dot{y}(t)=c(t)\left(b-A(x(t)+\dot{x}(t))-B(z(t)+\dot{z}(t))\right)\\[2ex]
x(0)=x^0\in \HH, z(0)=z^0 \in \GG, y(0)=y^0 \in \mathcal{K},
\end{cases}
\end{equation*}
where $c(t)>0$ for all $t \in [0,+\infty)$.

According to Theorem \ref{PAMADS:th:convergence} (for $L_{h_1}=L_{h_2}=0$), the generated trajectories converge  weakly to a saddle point of the Lagrangian, if we choose the map $c(t)$ as in this theorem and if there exists $\beta >0$ such that $B^*B\in \mathcal{P}_{\beta}(\mathcal{G})$.

\section{Conclusions and perspective}

In this paper we introduced and investigated a dynamical system which generates three trajectories in order
to approach the set of saddle points of the Lagrangian associated to a structured convex optimization problem with
linear constraints. Under appropriate conditions we showed that the systems is well-posed. The asymptotic
analysis is derived in the framework of Lyapunov analysis by finding an appropriate energy functional. The discretization of the considered dynamics is related to the Proximal AMA \cite{bibo19} and AMA \cite{tseng91} numerical schemes.

Let us mention some open questions as future research directions:

(i) Investigate convergence rates for the trajectories and also for the function values along the orbits. Notice
that in our setting $f$ is strongly convex and this might induce some rates. For the AMA algorithm
in \cite{tseng91} there are some results related to rates.

(ii) Consider second order dynamical systems in order to accelerate the convergence of the trajectories.
This would induce inertial terms in the discretized counterparts of the dynamics. For optimization
problems involving compositions with linear operators this is not a trivial task. We mention here
the paper of Attouch \cite{att}, where the starting point is a second order dynamics with vanishing damping for
monotone inclusion problems. The discretization leads to Proximal ADMM algorithms with momentum.
For an accelerated AMA numerical scheme we refer to \cite{g2014}.

(iii) The aim would be to conduct more involved numerical experiments related to optimization problems. More precisely, consider discretizations with variable step sizes in order to derive more general numerical schemes. This, in combination with different choices of the time varying positive semidefinite operators $M_1$ and $M_2$,  could have a great impact on the theoretical results and experiments.

\end{document}